\documentclass[11pt,reqno]{article}   
\usepackage{fullpage}

\usepackage[unicode=true]{hyperref}
\usepackage{amsmath}
\usepackage{cite}
\usepackage{amsfonts}
\usepackage{amssymb}
\usepackage{amsthm}
\usepackage{authblk}
\usepackage[pdftex]{color,graphicx}
\usepackage{adjustbox}
\usepackage{comment}
\usepackage{bbold}
\usepackage{tikz} 
\usetikzlibrary{decorations.pathreplacing,angles,quotes}
\usetikzlibrary{matrix,arrows,decorations.pathmorphing}
\usepackage{tikz-cd}
\usetikzlibrary{arrows}
\usetikzlibrary{decorations.markings}
\usepackage[all]{xypic}
\usepackage{bbold}
\usepackage{diagbox}
\usepackage{caption}
\usepackage{subcaption}
\usepackage{pdfpages} 
\usepackage{blkarray}
\usepackage{centernot}
\usepackage{mathtools}
\usepackage{stmaryrd}
\usepackage{soul}  
\usepackage{mypack-aziz} 

\definecolor{bluegray}{rgb}{0.4, 0.6, 0.8}
\definecolor{turquoise}{rgb}{0.2, 0.7, 0.6}

\usepackage{tikz-cd}

\begin{document}
	
	%
	%
	%
	\title{Simplicial distributions, convex categories and contextuality}
	
	\author{Aziz Kharoof\footnote{aziz.kharoof@bilkent.edu.tr} }
	\author{Cihan Okay\footnote{cihan.okay@bilkent.edu.tr}}
	\affil{Department of Mathematics, Bilkent University, Ankara, Turkey}
	
	\maketitle
	
\begin{abstract}
The data of a physical experiment can be {represented} as a presheaf of probability distributions. A striking feature of quantum theory is that those probability distributions obtained in quantum mechanical experiments do not always admit a joint probability distribution, a celebrated observation due to Bell. Such distributions are called contextual. Simplicial distributions are combinatorial models that extend presheaves of probability distributions by elevating sets of measurements and outcomes to spaces. Contextuality can be defined in this generalized setting.
This paper introduces the notion of convex categories to study simplicial distributions from a categorical perspective. Simplicial distributions can be given the structure of a convex monoid, a convex category with a single object, {when the outcome space has the structure of a group}. We describe contextuality as a monoid-theoretic notion {by introducing} a weak version of invertibility for monoids. 
Our main result is that a simplicial distribution is noncontextual if and only if it is weakly invertible. Similarly, strong contextuality and contextual fraction can be {characterized} in terms of invertibility in monoids. Finally, we show that simplicial homotopy can be used to detect extremal simplicial distributions refining the earlier methods based on \v Cech cohomology and the cohomology of groups.

\end{abstract}

	\tableofcontents

	
\section{Introduction} \label{cint}

Physical experiments are probabilistic: When a measurement is performed, the corresponding outcome occurs with a certain probability.  
Quantum theory
comes with an additional constraint prohibiting certain measurements from being performed jointly. 
Therefore in a quantum mechanical experiment, the data that describes the outcome probabilities consists of a family of probability distributions indexed over subsets of measurements that are allowed to be performed jointly.
More precisely, this family is a presheaf of probability distributions since the restriction of probabilities to a smaller subset of measurements, also called marginalization, is compatible.
A striking phenomenon in physics, known as Bell's nonlocality and its generalization called contextuality, can be expressed as the nonexistence of a joint probability distribution over the set of all measurements that marginalizes to the distributions of the restricted set of measurements obtained from the experiment.
Such a joint distribution always exists in classical theories.
In particular,  a joint distribution {provides a model where all}
measurement outcomes are assigned before the measurement takes place, and the measurement probabilities are obtained by considering all such global assignments with a certain probability.
It is a celebrated result of Bell \cite{bell64} that in quantum theory, the joint distribution does not always exist, i.e., there are contextual families of distributions.
Another celebrated result due to Kochen--Specker \cite{KS67} proves a similar result by showing the impossibility of the global assignments of outcomes in quantum theory. The latter demonstrates a stronger version of contextuality.

There are various approaches to formalizing contextuality.
The presheaf approach is introduced in \cite{abramsky2011sheaf}.
The ingredients in this approach are (1) a simplicial complex $\Sigma$ whose vertices represent the set of all measurements, and its simplices are those that can be jointly performed, and (2) a finite set of outcomes for the measurements. 
For the outcome set, we don't lose any generality by considering the ring $\ZZ_d$ of integers modulo $d$.
A distribution on $(\Sigma,\ZZ_d)$ is a presheaf of distributions, i.e., a family $(p_\sigma)_{\sigma\in \Sigma}$ together with a compatibility condition.
Each  $p_\sigma$ is a distribution on the set of all  functions $\sigma \to \ZZ_d$.
More formally, let $D_R$ denote the distribution monad on the category of sets \cite{Jacobs_2010}, where $R$ is a commutative semiring.
Then $p_\sigma$ belongs to $D_R(\ZZ_d^\sigma)$ where $R=\RR_{\geq 0}$, the semiring of nonnegative reals.
The presheaf approach uses tools from \v Cech cohomology to study contextuality. 
Furthermore, this systematic study introduces degrees of contextuality, such as strong contextuality and
the more refined measure of contextuality known as the contextual fraction.
Another approach to contextuality is the topological approach of \cite{Coho} based on techniques from the cohomology of groups.
This approach introduces cohomology classes that can detect strong contextuality but fail to capture weaker versions, e.g., the famous Clauser--Horne--Shimony--Holt (CHSH) scenario \cite{chsh69}.
In \cite{okay2022simplicial} both approaches are unified under the theory of simplicial distributions.
This theory goes beyond the usual assumption that measurements and outcomes are represented by finite sets. In this framework, one can study distributions on spaces of measurements and outcomes, where a space is represented by a simplicial set.
Simplicial sets are combinatorial objects more expressive than simplicial complexes. They are fundamental objects in modern homotopy theory \cite{goerss2009simplicial}.
A simplicial distribution is defined on a pair $(X,Y)$ of simplicial sets, where $X$ represents the measurements and $Y$ the outcomes.
The distribution monad can be elevated to a monad on the category of simplicial sets. Given $Y$, one can define another simplicial set $D_R(Y)$ whose simplices are distributions on the set of simplices of $Y$.
A simplicial distribution is a morphism of simplicial sets
$$
p:X\to D_R(Y).
$$
In this paper, we study simplicial distributions from a categorical perspective.

For a semiring $R$ the algebras of the distribution monad $D_R:\catSet\to \catSet$ are called $R$-convex sets. This notion generalizes the usual notion of convexity for $R=\RR_{\geq 0}$. 
Let $\catConv_R$ denote the category of $R$-convex sets.
Every monad has an associated Kleisli category.
In the case of $D_R$, the morphisms of the Kleisli category $\catsSet_{D_R}$ are the simplicial distributions, i.e., for simplicial sets $X,Y$ the set $\catSet_{D_R}(X,Y)$ of morphisms is given by simplicial set morphisms $X\to D_R(Y)$.
Let $s\catConv_R$ denote the category of simplicial $R$-convex sets. The main examples of such simplicial objects are $D_R(Y)$ for some simplicial set $Y$. 

\begin{prop}\label{intro-Hommm}
The functor sending a pair $(X,Y)$ of simplicial sets to the set $\St(X,Y)$ of simplicial set morphisms restricts
 to a functor
$$
\St(-,-): \St^{op} \times s\catConv_R \to \catConv_R
$$
\end{prop}

\noindent
The main application of this result is to the set of simplicial distributions. By this result the set $\catsSet(X,D_R(Y))$ is an $R$-convex set.
Contextuality for simplicial distributions is defined using a comparison map
$$
\Theta_{X,Y}: D_R(\catsSet(X,Y)) \to \catsSet(X,D_R(Y)).
$$
Under this map, the delta distribution at a simplicial set morphism $\varphi:X\to Y$ is sent to the simplicial distribution given by the composition $ X\xrightarrow{\varphi} Y \xrightarrow{\delta_Y} D_R(Y)$, called a deterministic distribution.
Since the target is $R$-convex, $\Theta_{X,Y}$ is the unique extension to the free $D_R$-algebra, the domain of the map.
A simplicial distribution is called contextual if it lies in the image of $\Theta_{X,Y}$; otherwise called noncontextual.
We show in Section \ref{LowDimm} how to realize a presheaf of distributions as a simplicial distribution.
In Theorem \ref{thm:ContextualSheaf} we show that the notion of contextuality for simplicial distributions specializes to the notion for presheaves of distributions originally introduced in \cite{abramsky2011sheaf}.

A category-theoretic point of view suggests lifting our analysis from the level of morphism sets to the level of categories.
For this, we introduce the notion of convex categories.
The first step is to lift   $D_R$ to a monad on the category $\catCat$ of (locally small) categories (Corollary \ref{DRCat}). 
Then we define an $R$-convex category as a $D_R$-algebra in $\catCat$.
Every category enriched over the category of $R$-convex sets is an $R$-convex category. However, the converse does not always hold. The prominent example of a convex category is the Kleisli category $\catSet_{D_R}$, and its simplicial version $\catsSet_{D_R}$, which are not enriched over $\catConv_R$.
We can think of $\Theta_{X,Y}$ assembled into a morphism in $\catConvCat_R$ from the free $R$-convex category to the Kleisli category, both of which obtained from the category of simplicial sets:
$$
\Theta: D_R(\catsSet) \to \catsSet_{D_R}.
$$ 
{For outcome spaces which also have a group structure the convex set of simplicial distributions can   be given a monoid structure.
Our categorical framework combined with this monoid structure has interesting applications connecting contextuality to a weak notion of invertibility in convex monoids.}
As we show in Section \ref{LowDimm} any presheaf of distribution can be realized as a simplicial distribution
$$
p: X\to D_R(N\ZZ_d),
$$
where $N\ZZ_d$ is the nerve of the additive group $\ZZ_d$.
The simplicial set $N\ZZ_d$ is, in fact, a simplicial group.
The group structure on the nerve induces a monoid structure on the convex set $\catsSet(X,D_R(N\ZZ_d))$ of simplicial distributions. 
We remark that this monoid structure is new and has not been investigated in the study of contextuality.
More precisely,  $\catsSet(X,D_R(N\ZZ_d))$ is a convex monoid, i.e., a convex category with a single object.
For a convex monoid $M$, with the map $\pi^M:D_R(M)\to M$ giving the $D_R$-algebra structure, an element $m\in M$ is called weakly invertible if it lies in the image of the composite {$D_R(M^*) \xhookrightarrow{D_R(i_M)} D_R(M)\xrightarrow{\pi^M} M$, where $i_M:M^*\hookrightarrow M$} is the inclusion of the units.
Our main result connects weak invertibility to noncontextuality.

\begin{thm}\label{intro-Theo2}
Let $R$ be a 
zero-sum-free, integral semiring. 
Given a simplicial set $X$ and a simplicial group $Y$, a distribution  $p \in \St(X,D_R(Y))$ is noncontextual if and only if $p$ is weakly invertible.  
\end{thm}

Moreover, we introduce the notion of strong invertibility, a monoid-theoretic description of strong contextuality. We also introduce a degree of invertibility called invertible fraction which generalizes  the notion of noncontextual fraction introduced in \cite{abramsky2011sheaf,Abramsky_2017}.
Simplicial distributions are formulated using the theory of simplicial sets.
It is a natural question to understand the role of homotopy in the context of simplicial distributions.
Embedding presheaves of distributions to our simplicial framework makes homotopical tools available for the study of contextuality.
In Corollary \ref{cor:homotopy} we show that simplicial homotopy can be used to detect extremal distributions, a question of fundamental importance in the study of polytopes of distributions; see \cite{pitowsky1989quantum,barrett2005nonlocal,
jones2005interconversion,abramsky2016possibilities}.

Our paper is organized as follows: In Section \ref{sec:alg-over-monad} we recall basics from convex sets and in Section \ref{sec:simp-conv-set} we introduce simplicial distributions. Proposition \ref{intro-Hommm} (Proposition \ref{Hommm}) is proved in Section \ref{sec:simp-dist-as-conv-set}.  In Section \ref{LowDimm}, we show how to describe a presheaf of distributions as a simplicial distribution.
In this section, we also provide examples of simplicial distributions, such as the CHSH scenario (Example \ref{Ex: CHSH}). Convex categories are introduced in Section \ref{subsec:conv-cat}. In Section \ref{sec:kleisli-convex-cat} we show that the Kleisli category $\catsSet_{D_R}$ is a convex category.
Weak invertibility, strong invertibility and invertible fraction are introduced in Sections \ref{sec:weak-invertibility}, \ref{sec:strong-invertibility} and \ref{sec:invertible-fraction}; respectively. Our main result Theorem \ref{intro-Theo2} (Theorem \ref{Theo2}) is proved in Section \ref{sec:cont-weak-inv}. The relationship between strong invertibility and strong contextuality is studied in Section \ref{sec:strong-cont-inv}.
Extremal simplicial distributions and the role of simplicial homotopy are discussed in Section \ref{sec:extremal}.

\paragraph{Acknowledgments.}
This  work is supported by the Air Force Office of Scientific Research under award
number FA9550-21-1-0002. 	
	
	\section{Simplicial distributions}\label{sec:SimpDist}

{
{Throughout the paper $R$ will denote a commutative semiring.}  
In this section, we introduce simplicial distributions \cite{okay2022simplicial} {defined over the semiring $R$}.
These objects describe distributions on a simplicial set 
parametrized by another simplicial set.
Our main result is that  the set of simplicial distributions constitute an $R$-convex set, in the sense that it is an algebra over the distribution monad. In practice, simplicial distributions come from presheaves of distributions, introduced in \cite{abramsky2011sheaf}. We describe how to embed the theory of presheaves of distributions into the theory of simplicial distributions.  
}

	\subsection{Algebras over a monad}\label{sec:alg-over-monad}

	We recall some basic facts about algebras over a monad from \cite[section VI]{Mac_Lane_1978}. 
	A {\it monad} on a category $\catC$ is a functor $T:\catC\to \catC$ together with natural transformations $\delta: \idy_\catC \to T$ and $\mu:T^2 \to T$ satisfying $\mu \circ T\mu ={\mu\circ \mu T}$ and $\mu\circ \delta T = \mu\circ T\delta = {\idy_T}$.
	A {\it $T$-algebra} consists of an object $X$ of
	$\catC$ together with a  morphism $\pi:T(X) \to X$ of $\catC$  such that the following diagrams commute  
	\begin{equation}\label{AAA}
		\begin{tikzcd}
			X \arrow[d,"\delta_X"'] \arrow[r,equal] & X \\
			T(X) \arrow[ru,"\pi"']
		\end{tikzcd}
		\;\;\;\;\;\;\;\;\;\;\;\;
		\begin{tikzcd}
			\hspace{1cm}
			T^2(X) \arrow[r,"T(\pi)"] \arrow[d,"\mu_X"'] & T(X) \arrow[d,"\pi"] \\
			T(X) \arrow[r,"\pi"] & X
		\end{tikzcd}
	\end{equation}
	A {\it morphism of $T$-algebras} is a morphism $f:X \to Y$ of $\catC$ such that 
$\pi^Y \circ T(f) =f\circ \pi^X$.
	The {\it category of $T$-algebras} will be denoted by $\catC^T$.
	The object $T(X)$ together with the structure morphism $\mu_{X}$ is called a {\it free $T$-algebra}. There is an adjunction $T:\catC \adjoint \catC^T:U$ where $T$ sends an object to the associated free $T$-algebra and $U$ is the forgetful functor. The {\it Kleisli category of $T$}, denoted by $\catC_T$, is the category whose objects are the same as the objects of $\catC$ and morphisms $X\to Y$ are given by $\catC(X,{T}(Y))$. In fact, $\catC_T$ is equivalent to the full subcategory of free $T$-algebras in $\catC^T$.  
See also\cite[Chapter 5]{riehl2017category} and \cite[Subsection 5.2]{perrone2019notes}.
\begin{remark}\label{AdjFormmm}
{\rm
Under the bijection $\catC^T(T(X),Y)\cong \catC(X,Y)$ a morphism $f:T(X)\to Y$ is sent to $f\circ \delta_X$. Conversely, under this isomorphism, a morphism $g:X\to Y$ is sent to $\pi^Y\circ T(g)$. 
}
\end{remark}
%
	
		

	\subsubsection{Convex sets}

Main examples of interest are the semiring of nonnegative reals $\RR_{\geq 0}$ and the Boolean algebra $\BB=\set{0,1}$. 
	The {\it distribution monad}  $D_R:\catSet \to \catSet$ is defined as follows:
	\begin{itemize}
		\item  {For a set $X$ the set $D_R(X)$ of $R$-distributions on $X$ consists of  functions $p:X\to R$ of finite support, i.e., $|\set{x\in X:\, p(x)\neq 0}|<\infty$, such that $\sum_{x\in X} p(x)=1$.}
		\item Given a function $f:X\to Y$ the function $D_R(f):D_R(X)\to D_R(Y)$ is defined by
		$$
		p \mapsto \left( y\mapsto \sum_{x\in f^{-1}(y)} p(x) \right). 
		$$ 
	\end{itemize}
	The structure maps of the monad are given as follows:
	\begin{itemize}
		\item  $\delta_X: X\to D_R(X)$ sends $x\in X$ to the delta distribution
		$$
		\delta^x(x') = \left\lbrace
		\begin{array}{ll}
			1 & x'=x,\\
			0 & \text{otherwise.}
		\end{array}
		\right.
		$$
		\item $\mu_X:D_R^2(X)\to D_R(X)$ sends a distribution $P$ to the distribution
		$$
		D_R(P)(x) = \sum_{p\in D_R(X)} P(p)p(x). 
		$$
	\end{itemize}

{	
\begin{defn}[\cite{jacobs2009duality}]
{\rm
A $D_R$-algebra in the category of sets is called an {\it $R$-convex set}. A morphism of $R$-convex sets is given by a morphism of $D_R$-algebras. We will denote the category of $R$-convex sets by $\catConv_R$.
}
\end{defn}	
}	

{In the case of $R=\RR_{\geq 0}$ the notion of $\RR_{\geq 0}$-convexity coincides with the usual notion of convexity, as we  recall from \cite{Jacobs_2010}  in Proposition \ref{Conv=ConvR} below. 
}


	


	\subsubsection{Real convex sets}

	
{We recall the definition of real convex sets from \cite[Definition 3]{Jacobs_2010}.}	
	A \emph{real convex set}  consists of a set X together with a ternary operation
	$\Span{-,-,-} : [0, 1] \times X \times X \to X$ 
	satisfying the following requirements, for
	all $\alpha, \beta \in [0, 1]$ and $x, y, z \in X$:
	\begin{enumerate}
		\item $\Span{\alpha,x,y}=\Span{1-\alpha,y,x}$.
		\item  $\Span{\alpha,x,x}=x$.
		\item   $\Span{0,x,y}=y$.
		\item   If $\alpha +(1-\alpha)\beta \neq 0$, then 
		$$
		\Span{\alpha,x,\Span{\beta,y,z}}=
		\Span{\alpha +(1-\alpha)\beta,\Span{\frac{\alpha}{\alpha +(1-\alpha)\beta},x,y},z}.
		$$
	\end{enumerate} 
	Given a real convex set $X$ it is sometimes more convenient to use the notation 
	$\alpha x +(1-\alpha) y$ instead of $\Span{\alpha,x,y}$.
{	A morphism of convex sets
	is
	an  function $f : X \to Y$ satisfying 
	%
$$
	f(\alpha x +(1-\alpha) x') =\alpha f(x) +(1-\alpha) f(x').
$$	}
	for all $\alpha \in [0, 1]$ and $x, x' \in X$. This yields the category of real convex sets denoted by
	$\catConv$.
  
	\begin{prop}[\cite{Jacobs_2010}]
	\label{Conv=ConvR}
		The category $\catConv$ is isomorphic to the category 
		$\catConv_{\rr_{\geq 0}}$. Under this isomorphism a real convex set $X$ is sent to the $\RR_{\geq 0}$-convex set $(X,\pi^X)$ where
		$$
		\pi^{X}: D_{\rr \geq 0}(X) \to X,\;\;\;\;\pi^X(P)=\sum_{x \in X} P(x)x.
		$$
		Conversely, an  $\rr_{\geq 0}$-convex set $(X,\pi^X)$ is sent to the real convex set where
		$$
		\Span{\alpha,x,y}=\pi^X(\alpha \delta^x + (1-\alpha)\delta^y).
		$$
\end{prop}

	%
	
	%
	%

	\subsection{Simplicial convex sets} \label{sec:simp-conv-set}

{We begin by introducing simplicial sets. Our main reference is  is \cite{goerss2009simplicial}.}	
	The {\it simplex category} $\catDelta$ consists of
	\begin{itemize}
		\item the objects $[n]=\set{0,1,\cdots,n}$ for $n\geq 0$, and
		\item the morphisms $\theta:[m]\to [n]$ given by order preserving functions.
	\end{itemize}
	A {\it simplicial set} is a functor $X:\catDelta^\op\to \catSet$. {The set of $n$-simplices is usually denoted by $X_n=X([n])$. 
An object $[m]$ in the simplex category gives a simplicial set $\Delta[m]$ whose set of $n$-simplices is given by  $\catDelta([n],[m])$. It is a well-known fact that every simplicial set can be obtained by ``gluing", or more formally as a colimit of, simplicial sets of this form.}
A {\it morphism of simplicial sets} is a natural transformation $f:X\to Y$ between the functors. We will write  $\catsSet$ for the category of simplicial sets.
	This definition can be extended to an arbitrary category $\catC$ and the resulting category is denoted by $s\catC$.
	
	
	\begin{prop}\label{SimMon}
		A monad $(T,\delta,\mu)$ on  $\catC$ extends to a monad $T:s\catC\to s\catC$ by applying $T$ degree-wise: 
		\begin{itemize}
			\item $T(X)$ is the simplicial object with $(TX)_n = T(X_n)$ and the simplicial structure maps are given by
			$d_i^{T(X)} = T(d_i^X)$  and $s_i^{T(X)} = T(s_i^X)$.
			\item $T(f): T(X)\to T(Y)$ in degree $n$ is given by  $T(f)_n = T(f_n)$.
			\item The monad structure maps $\delta$ and $\mu$ are defined by $(\delta_X)_n=\delta_{X_n}$ and $(\mu_X)_n=\mu_{X_n}$.  
		\end{itemize}
		Moreover, we have $s(\catC^T)
		\cong (s\catC)^T$.
	\end{prop}

Our main interest is the extension of the distribution monad to the category of simplicial sets, i.e., the functor 
$D_R:\catsSet \to \catsSet$. The associated category of $D_R$-algebras will be denoted by $s\catConv_R$.

\subsubsection{Simplicial distributions}

{Simplicial distributions are first introduced in  \cite{okay2022simplicial}. In this section we recall the basic definitions.} 
Let $\catsSet_{D_R}$ denote the Kleisli category of the distributions monad (acting on the category of simplicial sets). Its objects are simplicial sets and morphisms {between two simplicial sets $X$ and $Y$} are given by simplicial set morphisms of the form $p:X\to D_R(Y)$.  
 
	\begin{defn}\label{def:SimpDist}
		{\rm		
		Morphisms $\catsSet_{D_R}(X,Y)$ of the Kleisli category are called {\it simplicial distributions} on  $(X,Y)$.  
When the underlying semiring is $\RR_{\geq 0}$ we will call these morphisms simplicial probability distributions. For notational convenience  we write $p_\sigma$, where   $\sigma\in X_n$, for the distribution $p_n(\sigma)\in D_R(Y_n)$.  
		}
	\end{defn}
	
	There is a comparison map
	\begin{equation}\label{eq:Theta}
		\Theta_{X,Y}: D_R(\catsSet(X,Y)) \to \catsSet(X,D_R(Y))
	\end{equation}
	defined as follows: Let $x:\Delta[n]\to X$ be an $n$-simplex of $X$. We can construct a commutative diagram 
	$$
	\begin{tikzcd}
		D_R(\catsSet(X,Y)) \arrow{d}{D_R(x^*)} \arrow[r,dashed,"\Theta_{X,Y}"] &   \catsSet(X,D_R(Y)) \arrow{d}{x^*} \\
		D_R(\catsSet(\Delta[n],Y)) \arrow[r,"\Theta_x"] &   \catsSet(\Delta[n],D_R(Y))   
	\end{tikzcd}
	$$
	where the vertical maps are induced by $x$ and the bottom horizontal map is given by the identity map $D_R(Y_n)\to (D_RY)_n$.  
	The top horizontal map exists since $X$ is a colimit of its simplices. 
{The $\Theta$-map  gives rise to an important definition.}
	
	\begin{defn}
	\label{def:contextual-morphism}	
		{\rm
A simplicial distribution $p:X\to D_RY$ is called {\it contextual} if $p$ does not lie in the image of $\Theta_{X,Y}$. Otherwise, it is called {\it noncontextual}. When we want to refer to the semiring $R$  we say $R$-contextual, or $R$-noncontextual. 
		}
	\end{defn}
	
	The map $\Theta_{X,Y}$ can be given a more explicit description.
	
	\begin{prop}\label{pro:alternative-Theta}
		For $d\in D_R(\catsSet(X,Y))$ and $x\in X_n$ we have 
		$$
		\Theta (d)_n(x)= \sum_{\varphi \in \catsSet(X,Y)} d(\varphi) \delta^{\varphi_n(x)}.
		$$
That is, for $y\in Y_n$ we have
$
\Theta (d)_n(x):y \mapsto \sum_{\varphi_n(x)=y} d(\varphi).
$
\end{prop}

	In Corollary \ref{ThetaUniqqq} we will show that $\Theta$ is the unique map in $\catConv_R$ that makes the following diagram 
	commute 
	\begin{equation}\label{THETAA}
		\begin{tikzcd}
			\catsSet(X,Y) \arrow[d,"\delta_{\catsSet(X,Y)}"']  \arrow[r,"(\delta_Y)_*"] & \catsSet(X, D_R(Y)) \\
			D_R( \catsSet(X,Y) ) \arrow[ru, "\Theta"']
		\end{tikzcd}
	\end{equation}
{Alternatively,  noncontextual distributions can be described as convex mixtures of distributions in the image of $(\delta_Y)^*$.} 
 
\begin{defn}\label{def:deterministic}
{\rm
A simplicial distribution on $(X,Y)$ is called {\it deterministic} if it lies in the image of 
$(\delta_Y)_\ast:\St(X,Y) \to \St(X,D_R(Y))$.
{In this case the resulting distribution is denoted by $\delta^\varphi=(\delta_Y)_*(\varphi)$.}
}
\end{defn}

{ 
\begin{ex}\label{ex:simplex}
{\rm
Let $X=\Delta[n]$ and $Y$ be an arbitrary simplicial set. There is a bijection between the set $\catsSet(\Delta[n],D_R(Y))$ of simplicial distributions and the set $D_R(Y)_n=D_R(Y_n)$ of $n$-simplices. Similarly, $\catsSet(\Delta[n],Y)$ can be identified with $Y_n$. 
Then the $\Theta$-map is the identity map and every simplicial distribution is noncontextual. 
}
\end{ex}
} 

{
Next, we introduce a stronger version of contextuality. 
}

\begin{defn}\label{suppp}
{\rm	
The {\it support of a simplicial distribution}
	$p:X\to D_R Y$ 
is defined by
	$$
	\supp(p) =\{ \varphi \in \St(X ,Y)  :\, p_n(x)(\varphi_n(x)) \neq 0,\; \forall x\in X_n,\, n \geq 0  \}.
	$$
	We say $p$ is {\it strongly contextual} if $\supp(p)$ is empty.
}
\end{defn}

\begin{defn}[\cite{jacobs2009duality}]
\label{def:ZeroFreeSum}
A semiring is called {\it zero-sum-free} if $a+b=0$ implies $a=b=0$ for all $a,b\in R$.
\end{defn}

\begin{prop}
Let $R$ be
a zero-sum-free semiring.  
If a simplicial distribution
$p$ is strongly contextual 
then it is contextual.
\end{prop}

\begin{proof}
	Suppose that $p$ is noncontextual, that is there exists 
	$d \in D_R(\St(X,Y))$ such that
	$$
	p_n(x)=\sum_{\varphi \in \St(X,Y)} d(\varphi) \delta^{\varphi_n(x)}.
	$$
	Then $\psi:X\to Y$ with $d(\psi)\neq 0$ will belong to the support of $p$ since
	$$
	p_n(x)(\psi_n(x)) =  \sum_{\varphi \in \St(X,Y)} 
	d(\varphi) \delta^{\varphi_n(x)}(\psi_n(x)) =d(\psi) + 
	\sum_{\varphi  \neq \psi } 
	d(\varphi) \delta^{\varphi_n(x)}(\psi_n(x))  \neq 0.
	$$
\end{proof}   

For simplicial probability distributions one can introduce a degree of contextuality generalizing the notion introduced in \cite[Subsection 6.1]{abramsky2011sheaf}.

\begin{defn}\label{CFFF}
{\rm
The {\it noncontextual fraction} of a simplicial probability distribution $p\in \St(X,D(Y))$, which is denoted
	by $\NCF(p)$, is defined to be the supremum of $\alpha\in [0,1]$ such that
$$
p=\alpha q+ (1-\alpha)s
$$
where $q$ and $s$ run over   simplicial probability distributions with $q$ noncontextual. 
	%
The {\it contextual fraction} of $p$ is defined to be $\CF(p)= 1-\NCF(p)$. 
}	 
\end{defn}

	\subsection{Simplicial distributions as a convex set} \label{sec:simp-dist-as-conv-set}
	
		
		


In this section we will show that $\catsSet(X,Y)$, {where $Y$ is an object of $s\catConv_R$,} is an $R$-convex set (Proposition \ref{Hommm}) and prove a uniqueness result about the comparison map $\Theta_{X,Y}$ (Proposition \ref{ThetaUniqqq}). 
For simplicity of notation we will write $U=\St(X,Y)$, especially when this set appears in the structure maps for convexity.

\begin{lem}\label{lem:two-eqs}
{Let $Y$ be an object of $s\catConv_R$.}
For $x \in X_n$, consider the map
$$
e_x : \St(X,Y) \to Y_n
$$ 
defined by $e_x(\varphi) = \varphi_n(x)$.
Given $Q \in D_R(D_R(\St(X,Y)))$ let $P=D_R(D_R(e_x))(Q)$ and
define $\pi^{U}$ to be the composite
$$ 
D_R(\St(X,Y)) \xrightarrow{\Theta_{X,Y}} \St(X,D_RY) \xrightarrow{(\pi^Y)_*}  \St(X,Y).
$$
Then $P \in D_R(D_R(Y_n))$
satisfies the following equations:
\begin{equation}\label{Eq1}
D_R(\pi^{Y_n})(P)= \Theta_{X,Y}(D_R(\pi^U)(Q))_n(x)
\end{equation}
and
\begin{equation}\label{Eq2}
\mu_{Y_n}(P)=\Theta_{X,Y}(\mu_{U}(Q))_n(x).
\end{equation}
\end{lem}		
\begin{proof}
To prove Equation (\ref{Eq1}) we first verify that 
\begin{equation}\label{eq:middle-step}
\pi^{Y_n} \circ D_R(e_x) = e_x \circ \pi^{U}.
\end{equation}
For $q \in D_R(U)$ we have
\begin{equation}\label{BBB}
D_R(e_x)(q)=\sum_{\varphi \in U}q(\varphi) \delta^{e_x(\varphi)}
=\sum_{\varphi \in U} q(\varphi) \delta^{\varphi_n(x)}
=\Theta_{X,Y}(q)_n(x).
\end{equation}
Using this we obtain 
$$
\begin{aligned}
\pi^{Y_n}(D_R(e_x)(q)) &=
\pi^{Y_n}\left(\Theta_{X,Y}(q)_n(x)\right)\\
&= \left(\pi^{Y}\circ (\Theta_{X,Y}(q))\right)_n(x) \\
&=e_x\left(\pi^{Y}\circ (\Theta_{X,Y}(q))\right) \\
&=e_x\circ (\pi^Y)_\ast \circ \Theta_{X,Y}(q) \\
&=e_x \circ \pi^{U}(q).
\end{aligned}
$$
Now using Equation (\ref{eq:middle-step}) and applying Equation (\ref{BBB}) to $q=D_R(\pi^{U})(Q)$ we obtain 
$$
\begin{aligned}
D_R(\pi^{Y_n})(P) &=D_R(\pi^{Y_n})\left(D_R(D_R(e_x))(Q)\right)\\
& =
D_R\left(\pi^{Y_n} \circ D_R(e_x)\right)(Q) \\
&=D_R(e_x \circ \pi^{U})(Q)\\
&=D_R(e_x)\left(D_R(\pi^{U})(Q)\right) \\
&=\Theta_{X,Y}\left(D_R(\pi^{U})(Q)\right)_n(x),
\end{aligned}
$$
which proves Equation (\ref{Eq1}).

Next we prove Equation (\ref{Eq2}):
For $ y\in Y_n$ we have
$$
\begin{aligned}
\mu_{Y_n}(P)(y) &=
\mu_{Y_n}(D_R(D_R(e_x))(Q))(y) \\
&=
\sum_{\tilde{q} \in D_R(Y_n)}D_R(D_R(e_x))(Q)(\tilde{q})\tilde{q}(y)\\
&=\sum_{\tilde{q} \in D_R(Y_n)}
\left(\sum_{q \in D_R(U):\,
D_R(e_x)(q)=\tilde{q}}Q(q)\right)\tilde{q}(y) \\
&=\sum_{\tilde{q} \in D_R(Y_n)} \,\,
\sum_{q \in D_R(U):\,
D_R(e_x)(q)=\tilde{q}}Q(q)D_R(e_x)(q)(y)\\
&=\sum_{q \in D_R(U)} Q(q)\left(\sum_{\varphi \in U:\,
e_x(\varphi)=y}q(\varphi)\right)\\
&=\sum_{q \in D_R(U)}~\sum_{\varphi \in U :\,
\varphi_n(x)=y}Q(q)q(\varphi) \\
&= \sum_{\varphi \in U:\,
\varphi_n(x)=y}\mu_{U}(Q)(\varphi)\\
&= \Theta_{X,Y}(\mu_{U}(Q))_n(x)(y).
\end{aligned}
$$
\end{proof}

\begin{prop}\label{Hommm}
The functor $\St(-,-) : \St^{op} \times \St \to \catSet$ restricts to a functor
$$
\St(-,-): \St^{op} \times s\catConv_R \to \catConv_R
$$
\end{prop}
\begin{proof}
Let $X$ be a simplicial set and $Y$ be an object of $s\Conv_R$. 
{Let $\pi^{U}$ be the map that defined in Lemma \ref{lem:two-eqs},} we want to prove that with this structure map
 $U=\catsSet(X,Y)$  
is an object of $\catConv_R$.
First we verify that $\pi^U \circ \delta_U = \Id_U$:
$$
\begin{aligned}
\pi^{U} \circ \delta_{U} &=
 (\pi^Y)_\ast \circ \Theta_{X,Y} \circ \delta_{U}\\
 &=
(\pi^Y)_\ast \circ(\delta_Y)_\ast \\
&= (\pi^Y \circ \delta_Y)_\ast\\
 &=(\Id_Y)_\ast \\
&=\Id_{U}.
\end{aligned}
$$
Next we show that $\pi^U \circ D_R(\pi_U) = \pi^U \circ \mu_U$, that is 
the following diagram commutes
$$
\begin{tikzcd}
D_R(D_R(U)) \arrow[rrr,"D_R(\pi^{U})"]
 \arrow[d,"\mu_{U}"']
 &&& D_R(U) \arrow[d,"\pi^{U}"] \\
D_R(U) \arrow[rrr,"\pi^{U}"] &&& U 
\end{tikzcd}
$$
For  $Q \in D_R(D_R(U))$ and $x\in X_n$ we have 
%
%
%
%
$$
\begin{aligned}
(\pi^{U}(D_R(\pi^{U})(Q)))_n(x)&=
((\pi^{Y})_\ast \circ \Theta_{X,Y} (D_R(\pi^{U  })(Q)))_n(x) \\
&=(\pi^{Y} \circ(\Theta_{X,Y} (D_R(\pi^{U  })(Q))))_n(x) \\
&=\pi^{Y_ n} \circ(\Theta_{X,Y} (D_R(\pi^{U  })(Q)))_n(x)  \\
&=\pi^{Y_n}(\Theta_{X,Y}(D_R(\pi^{U})(Q))_n(x)).
\end{aligned}
$$
Similarly  
$
(\pi^{U}(\mu_{U}(Q)))_n(x)
=\pi^{Y_n}(\Theta_{X,Y}(\mu_{U}(Q))_n(x)).
$
Therefore we need to prove that
\begin{equation}\label{eq:main}
\pi^{Y_n}(\Theta(D_R(\pi^U)(Q))_n(x))
=\pi^{Y_n}(\Theta(\mu_{U}(Q))_n(x)).
\end{equation}
Substituting Equations (\ref{Eq1}) and (\ref{Eq2}) from Lemma \ref{lem:two-eqs} into Equation (\ref{eq:main}) and using  $\pi^{Y_n}\circ D_R(\pi^{Y_n})=\pi^{Y_n}\circ \mu_{Y_n}$  gives the desired result.
It remains to  show that for $f:Y_1\to Y_2$ and $g:X_1\to X_2$ the induced maps are morphisms in $\catConv_R$. This follows from the naturality of $\Theta_{X,Y}$ in both $X$ and $Y$.

\end{proof}
%
	%
	
{For a simplicial set $Y$ observe  that $D_R(Y)$ is an object of $s\catConv_R$, hence Proposition \ref{Hommm}	applies.}
We employ $\catsSet(X,D_R(Y))$ 
with an $R$-convex set structure by defining 
\begin{equation}\label{eq:Rconv-structure-sSet}
\pi^{\St(X,D_R(Y))}=(\pi^{D_R(Y)})_\ast \circ \Theta_{X,D_R(Y)}.
\end{equation}
More explicitly, for $Q \in D_R(\catsSet(X,D_R(Y)))$ and $x\in X_n$ this gives
	\begin{equation}\label{piForm}
	\pi(Q)_n(x)=\sum_{p \in \catsSet(X,D_R(Y)) }Q(p)\,p_n(x). 
	\end{equation}
	%
%
\begin{prop}\label{ThetaUniqqq}
%
The map $\Theta_{X,Y}$ is the transpose
of $(\delta_Y)_\ast : \St(X,Y) \to \St(X,D_R(Y))$ in $\catConv_R$, {with respect to the adjunction $\catSet\adjoint \catConv_R$,} and it is given by the composite
\begin{equation}\label{eq:Thetapi}
\Theta_{X,Y}=\pi^{\St(X,D_R(Y))}\circ D_R((\delta_Y)_\ast).
\end{equation}
In particular, it is the unique map in $\catConv_R$ that makes Diagram (\ref{THETAA}) commutes.
\end{prop}
\begin{proof}
By the naturality of $\Theta_{X,-}$ we have the following commutative diagram:
\begin{equation}\label{NattD}
\begin{tikzcd}
D_R(\St(X,Y)) \arrow[rr,"D_R((\delta_Y)_\ast)"] \arrow[d,"\Theta_{X,Y}"] 
&& D_R(\St(X,D_R(Y))) \arrow[d,"\Theta_{X,D_R(Y)}"] \\ 
 \St(X,D_R(Y)) \arrow[rr,"(D_R(\delta_Y))_\ast"] && 
\St(X,D_R(D_R(Y))) 
\end{tikzcd}
\end{equation}
that is the following equation holds 
\begin{equation}\label{eq:Natt}
\Theta_{X,D_R(Y)} \circ D_R((\delta_Y)_\ast) = (D_R(\delta_Y))_\ast \circ \Theta_{X,Y}.
\end{equation}
We compose with $(\mu^Y)_\ast$ on both sides of Equation (\ref{eq:Natt}) to obtain 
\begin{equation}\label{eq:Naturalll}
(\mu^Y)_\ast \circ \Theta_{X,D_R(Y)} \circ D_R((\delta_Y)_\ast) 
=(\mu^Y)_\ast \circ (D_R(\delta_Y))_\ast \circ \Theta_{X,Y}.
\end{equation}
Now the last composite on the left-hand side can be rewritten as 
\begin{equation}\label{eq:NaturalllB}
(\mu^Y)_\ast \circ \Theta_{X,D_R(Y)}=
(\pi^{D_R(Y)})_\ast \circ \Theta_{X,D_R(Y)}=\pi^{\St(X,D_R(Y))}
\end{equation}
and similarly for the right-hand side we have 
\begin{equation}\label{eq:NaturalllA}
(\mu^Y)_\ast \circ (D_R(\delta_Y))_\ast= 
 (\mu^Y \circ D_R(\delta_Y))_\ast =(\Id_{D_R(Y)})_\ast =\Id_{\St(X,D_R(Y))}.
\end{equation}
which proves Equation (\ref{eq:Thetapi}).
By Remark \ref{AdjFormmm} we see that $\Theta_{X,Y}$ is the transpose
of $(\delta_Y)_\ast$ in $\catConv_R$. 
\end{proof}

	\subsection{Simplicial distributions from presheaves of distributions}\label{LowDimm}

	
	Simplicial {presheaf of distributions} can be constructed from the following data:
	\begin{itemize}
		\item A simplicial complex $\Sigma$.
		\item The ring $\ZZ_d=\set{0,1,\cdots,d-1}$ of integers mod $d\geq 2$.
	\end{itemize}
	We will write $\Sigma_0$ to denote the set of vertices of the simplicial complex. We will think of $\Sigma$ as a category $\catC_\Sigma$ whose objects are the simplices $\sigma \in \Sigma$ and morphisms are inclusions $\sigma \hookrightarrow \sigma'$.
	An element $p$ of the inverse limit  of the composite functor
	$$
	\catC_\Sigma \xrightarrow{\ZZ_d^-} \catSet \xrightarrow{D_R} \catSet
	$$
	is called  an {\it $R$-distribution on $(\Sigma,\ZZ_d)$} (also known as an {\it empirical model}). 
{Let $V\subset U$ be two subsets of $\Sigma_0$ and $i:V\hookrightarrow U$ denote the inclusion map.
For $p\in D_R(\ZZ_d^U)$ we will write $p|_V$ for the distribution $D_R(i^*)(p)\in D_R(\ZZ_d^V)$, where $i^*:\ZZ_d^U\to \ZZ_d^V$ is the restriction map.
Then,} a presheaf of distributions consists of a tuple $p=(p_\sigma)_{\sigma\in \Sigma}$ of distributions $p_\sigma\in D_R(\ZZ_d^\sigma)$ such that {
$$
p_\sigma |_{\sigma\cap \sigma'}= p_{\sigma'} |_{\sigma\cap \sigma'},\;\;\forall\sigma,\sigma'\in \Sigma,
$$}
	We write $D_R(\Sigma,\ZZ_d)$ 
for the set of $R$-distributions on $(\Sigma,\ZZ_d)$. 
For more details see \cite{abramsky2011sheaf}.
	
	For a set $U$ let $\Delta_U$ denote the simplicial set whose $n$-simplices are given by the set $U^{n+1}$ and 
	the simplicial structure maps are given by
	$$
	\begin{aligned}
		d_i(x_0,x_1,\cdots,x_n) &=( x_0,x_1,\cdots,x_{i-1},x_{i+1},\cdots,  x_n) \\
		s_j(x_0,x_1,\cdots,x_n) &=( x_0,x_1,\cdots,x_{j-1},x_j,x_{j},x_{j+1},\cdots,  x_n).
	\end{aligned}
	$$
For the next definition we will use the following notation: For $p:X\to D_RY$ and $\sigma\in X_n$ we write $p_\sigma$ for the distribution on $D_RY_n$.		
	
	\begin{construction}[Realization]\label{cons:ZeroDim}{\rm
			Given $(\Sigma,\ZZ_d)$ we consider a pair of simplicial sets $(X_\Sigma,Y)$ where $X_\Sigma$ is the subsimplicial set of $\Delta_{\Sigma_0}$ whose $n$-simplices are given by
			$$
			(X_\Sigma)_n = \set{(x_0,x_1,\cdots,x_n)\in \Sigma_0^{n+1}:\, \set{x_0,x_1,\cdots,x_n}\in \Sigma}.
			$$    
			We define a function 
			$$
			\real:D_R(\Sigma,\ZZ_d) \to \catsSet(X_\Sigma, D_R(\Delta_{\ZZ_d}))
			$$
			by sending $p=(p_\sigma)_{\sigma\in \Sigma}$ to the simplicial distribution
			$
			\real(p): X_\Sigma \to D_R(\Delta_{\ZZ_d})
			$
			defined by
			$$
			\real(p)_{(x_0,x_1,\cdots,x_n)}(a_0,a_1,\cdots,a_n) =
 p_{\{x_0,x_1,\cdots,x_n\}}(s:x_i\mapsto a_i).
			$$
		}
	\end{construction}
 


\begin{thm}\label{thm:ContextualSheaf}
			Let $p=(p_\sigma)_{\sigma\in \Sigma}$ be an $R$-distribution on 
			$(\Sigma,\ZZ_d)$.
			Then $\real(p)$ is {non}contextual if and only if there exists {$\tilde p\in D_R(\ZZ_d^{\Sigma_0})$ such that $\tilde p|_\sigma = p_\sigma$} for all $\sigma\in \Sigma$.
	\end{thm}
\begin{proof} 
For $\sigma\in \Sigma$ the composite map 
$$
\eta_\sigma: D_R(\ZZ^\sigma) \xrightarrow{\cong} D_R(\catsSet(\Delta_\sigma,\Delta_{\ZZ_d})) \xrightarrow{\Theta} \catsSet(\Delta_\sigma,D_R(\Delta_{\ZZ_d}))
$$
is injective. 
Choosing a well-ordering for $\Sigma_0$ the maps $\eta_\sigma$ can be assembled into the $\eta$ map in the following commutative diagram
$$
\begin{tikzcd}
D_R(\catsSet(X_\Sigma,\Delta_{\ZZ_d})) \arrow[d,"\cong"] \arrow[r,"\Theta"] & \catsSet(X_\Sigma,D_R(\Delta_{\ZZ_d})) \arrow[r,"\cong"] & \lim \catsSet(\Delta_\sigma,D_R(\Delta_{\ZZ_d})) \\
D_R(\ZZ_d^{\Sigma_0}) \arrow[r,"\tilde p\;\mapsto \;(\tilde p|_\sigma)"] & \lim D_R(\ZZ_d^\sigma) \arrow[ur,"\eta"',hook]
\end{tikzcd}
$$ 
This gives the desired result.
\end{proof}

\begin{rem}\label{rem:strong-fraction}
{\rm
The condition of being strongly contextual also simplifies in the case of $\real(p)$. The support of the distribution (Definition \ref{suppp}) becomes
$$
\supp(\real(p)) = \set{s\in \ZZ_d^{\Sigma_0}:\, p_\sigma(s|_{\sigma})\neq 0,\;\forall \sigma\in \Sigma }.
$$
This is precisely the definition of the support of a distribution in $D_R(\Sigma,\ZZ_d)$ \cite{abramsky2011sheaf}.

Observe that the map $\eta$ given in the proof of Theorem \ref{thm:ContextualSheaf} is a morphism of $R$-convex sets. In particular for $R=\RR_{\geq 0}$ this implies that the notion of contextual fraction for simplicial distributions (Definition \ref{CFFF}) coincides with the corresponding notion for distributions in $D(\Sigma,\ZZ_d)$ given in \cite{abramsky2011sheaf}.
}
\end{rem}

\begin{rem}
{\rm
The realization $\real$ defined in Construction \ref{cons:ZeroDim} has an ordered version. That is, given an ordered simplicial complex $\Sigma$, one can define $X_\Sigma^\ord$
as the subsimplicial set of $X_\Sigma$ consisting of ordered tuples of vertices that constitute a simplex. Then the $\eta$ map in the proof of Theorem \ref{thm:ContextualSheaf} is an isomorphism
$$
\eta: \lim D_R(\ZZ_d^\sigma) \xrightarrow{\cong} \catsSet(X_\Sigma^\ord,D_R(\Delta_{\ZZ_d}))
$$
For details see \cite[Theorem B.2]{okay2022simplicial}.
{In the examples below we will consider this ordered version since the corresponding pictures contain smaller number of simplices.}
}
\end{rem}

\begin{ex}\label{ex:edge-dist}
{\rm
{Let $\Sigma$ denote the simplicial complex with vertices $x,y$ and a single maximal simplex $\set{x,y}$.
An element of $D(\Sigma,\ZZ_2)$ is simply a distribution $p\in D(\ZZ_2^2)$. Let us write $p^{ab}=p(a,b)$ where $a,b\in \ZZ_2$.
With this notation $p$ can be conveniently represented as a box (table):
}
$$
\begin{tabular}{|c|c|}  
\hline
  & $y$  \\ 
\hline
$x$ & {\begin{tabular}{cc} $p^{00}$ & $p^{01}$ \\ $p^{10}$ & $p^{11}$ \\ \end{tabular}}  \\
\hline
\end{tabular} 
$$
{
Let us write $p_x=p|_{\set{x}}$ and $p_y=p|_{\set{y}}$. Then the compatibility relations become
\begin{equation}\label{eq:x-y}
\begin{aligned}
p_x^0 &= p^{00} + p^{01}\\
p_y^0 &= p^{00} + p^{10},
\end{aligned}
\end{equation}
which can be read off from the rows and the columns of the box.
Let us describe the corresponding  simplicial distribution.
The ordered version $X_\Sigma^\ord$ is isomorphic to $\Delta[1]$. Therefore $p$ gives a simplicial distribution of the form $\Delta[1]\to D(\Delta_{\ZZ_2})$, again denoted by $p$ for simplicity.
}
Let $\iota_1$ denote the unique nondegenerate $1$-simplex of $\Delta[1]$, i.e., the identity map $[1]\to [1]$. 
Then the value of the distribution $p_{\iota_1}$ {(see Definition \ref{def:SimpDist} for notation)} at the $1$-simplex $(a,b)$ of $\Delta_{\ZZ_2}$ is given by $p^{ab}$.
The two $0$-simplices $d_1(\iota_1)$ and $d_0(\iota_1)$ correspond to $x$ and $y$; respectively.
The conditions in Equation (\ref{eq:x-y}) are expressed as the simplicial relation   $d_ip_{\iota_1}=p_{d_i\iota_1}$.
} 
\end{ex}

\begin{example}\label{Ex: CHSH}
{\rm
A famous example, known as the 
CHSH scenario \cite{chsh69}, consists of  the simplicial complex on the vertex set $\Sigma_0=\set{x_0,x_1,y_0,y_1}$ determined by the maximal simplices
$$
\set{x_0,y_0},\; \set{x_0,y_1},\;\set{x_1,y_0},\; \set{x_1,y_1} .
$$ 
The resulting (ordered) simplicial complex is the boundary of a square:
$$ 
\includegraphics[width=.2\linewidth]{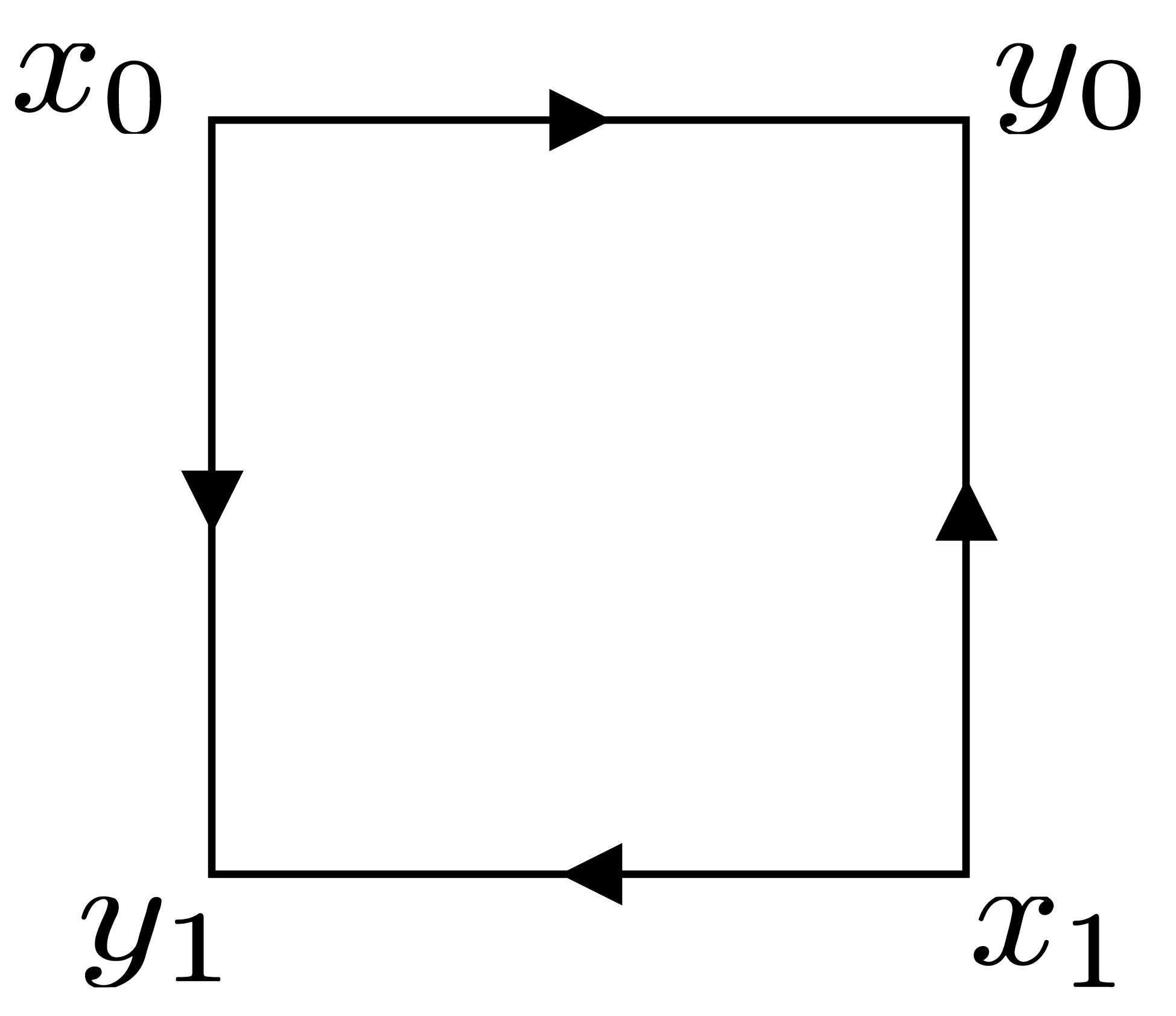}  
$$
A distribution $p=(p_\sigma)_{\sigma\in \Sigma}\in D(\Sigma,\ZZ_2)$ consists of distributions $p_{\set{x_i,y_j}}\in D(\ZZ_2^{\set{x_i,y_j}})$ together with the compatibility conditions imposed by the inverse limit. Writing $p_{x_iy_j}^{ab}$ for the probability $p_{\set{x_i,y_j}}(s)$, where $s:\set{x_i,y_j}\to \ZZ_2$ defined by $s(x_i)=a$, $s(y_j)=b$, these conditions can be expressed as
\begin{equation}\label{eq:chsh-conditions}
\sum_{a\in \ZZ_2} p_{x_iy_0}^{ab} = \sum_{a\in \ZZ_2} p_{x_iy_1}^{ab},\;\;\;\;
\sum_{b\in \ZZ_2} p_{x_0y_j}^{ab} = \sum_{b\in \ZZ_2} p_{x_1y_j}^{ab}. 
\end{equation}
{Again a convenient way to represent this data is to use a table of the form}
$$
\begin{tabular}{|c|c|c|c|c|} 
\hline
  & $y_0$ & $y_1$ \\ 
\hline
$x_0$ & {\begin{tabular}{cc} $p_{x_0y_0}^{00}$ & $p_{x_0y_0}^{01}$ \\ $p_{x_0y_0}^{10}$ & $p_{x_0y_0}^{11}$ \\ \end{tabular}} & {\begin{tabular}{cc} $p_{x_0y_1}^{00}$ & $p_{x_0y_1}^{01}$ \\ $p_{x_0y_1}^{10}$ & $p_{x_0y_1}^{11}$ \\ \end{tabular}} \\ 
\hline
$x_1$ & {\begin{tabular}{cc} $p_{x_1y_0}^{00}$ & $p_{x_1y_0}^{01}$ \\ $p_{x_1y_0}^{10}$ & $p_{x_1y_0}^{11}$ \\ \end{tabular}} & {\begin{tabular}{cc} $p_{x_1y_1}^{00}$ & $p_{x_1y_1}^{01}$ \\ $p_{x_1y_1}^{10}$ & $p_{x_1y_1}^{11}$ \\ \end{tabular}} \\  
\hline
\end{tabular} 
$$
The compatibility conditions in Equation (\ref{eq:chsh-conditions}) can be read off the table.
According to a celebrated theorem due to Fine \cite{fine1982hidden,fine1982joint}, a distribution $p\in D(\Sigma,\ZZ_2)$ is noncontextual if and only if the CHSH inequalities are satisfied. 
Moreover, it is well-known that $D(\Sigma,\ZZ_2)$ is a polytope with $16$ deterministic vertices and $8$ contextual vertices given by the  Popescu--Rohrlich (PR) boxes \cite{Popescu_1994}. 
}
\end{example}

Next, we describe another way of realizing presheaves of distributions where the target space is distributions on the nerve space $N\ZZ_d$. \
For a monoid $(M,\cdot)$ the nerve $N(M)$  is the simplicial set whose set of $n$-simplices is given by $M^n$ with the following simplicial structure:
$$
\begin{aligned}
d_i(m_1,m_2,\cdots,m_n) &= \left\lbrace \begin{array}{ll}
(m_2,m_3,\cdots,m_n)   &  i=0\\
(m_1,\cdots,m_i\cdot m_{i+1} ,\cdots,m_n) & 0<i<n\\
(m_1,m_2,\cdots,m_{n-1}) & i=n
\end{array}
\right. \\
s_j(m_1,m_2,\cdots,m_n) &= (m_1,\cdots,m_{j-1}, e_M,m_{j},\cdots,  m_n)\;\;\, 0\leq j\leq n,
\end{aligned}
$$
where $e_M$ is the identity element.
We will need the following adjunction 
\begin{equation}\label{eq:adjunction-Dec}
	\Delta[0]\ast -:\catsSet \adjoint \catsSet: \Dec^0	 
\end{equation}
where $(- \ast -)$ denotes the join of the simplicial sets and  $\Dec^0$ is the d\' ecalage functor \cite{stevenson2011d}. 
The simplicial set $\Dec^0 X$ is obtained from shifting the simplices of $X$ down by one degree, i.e., $(\Dec^0 X)_n = X_{n+1}$, and forgetting the first face and degeneracy maps. 

\begin{pro}\label{pro:Dec}
The adjunction given in Equation (\ref{eq:adjunction-Dec}) induces 
a commutative diagram of $R$-convex sets
$$
\begin{tikzcd}
D_R(\catsSet(\Delta[0] \ast X,Y )) 
\arrow[rr,"\Theta_{\Delta[0] \ast X,Y}"]
 \arrow[d,"\cong"] && \catsSet(\Delta[0] \ast X,D_RY) \arrow[d,"\cong"] \\
D_R(\catsSet(X,\Dec^0Y)) \arrow[rr,"\Theta_{X,\Dec^0(Y)}"] && \catsSet(X,D_R(\Dec^0Y))
\end{tikzcd}
$$ 
\end{pro}
\begin{proof}
For a simplicial map $f:\Delta[0] \ast X \to Y$, the transpose of the compositon $\Delta[0] \ast X \stk{f} Y \stk{\delta_Y} D_R(Y)$ under 
the adjunction that given in Equation (\ref{eq:adjunction-Dec}) is equal 
to the composition of the transpose of $f$ under the same adjunction 
with $\Dec^0(\delta_Y)$.
Observe also that the map $\Dec^0(\delta_Y)$ equal to the map 
$\delta_{\Dec^0Y}$, so we have the following diagram in $\St$
$$
\begin{tikzcd}
\catsSet(\Delta[0] \ast X,Y ) \arrow[r,"(\delta_Y)_\ast"] \arrow[d,"\cong"] & \catsSet(\Delta[0] \ast X,D_RY) \arrow[d,"\cong"] \\
\catsSet(X,\Dec^0Y) \arrow[r,"(\delta_{\Dec^0Y})_\ast"] & \catsSet(X,D_R(\Dec^0 Y))
\end{tikzcd}
$$
Using Proposition \ref{ThetaUniqqq} we get the result.
\end{proof}

There is an isomorphism of simplicial sets
$$
\Delta_{\ZZ_d} \xrightarrow{\cong} \Dec^0(N\ZZ_d)
$$
defined in degree $n$ by sending $(a_0,a_1,\cdots,a_{n})$ to the tuple $(a_0,a_1-a_0,\cdots,a_n-a_{n-1})$. 
Note that the isomorphism is preserved even after applying the $D_R$ functor since essentially d\' ecalage only shifts the dimension.
Therefore there is a bijection 
$$
\catsSet(X , D_R(\Delta_{\ZZ_d})) \cong 
\catsSet(\Delta[0] \ast X ,   D_R(N\ZZ_d))
$$
and Proposition \ref{pro:Dec} implies that the notion of contextuality for both scenarios $(\Delta[0] \ast X , N\zz_d)$ and 
$(X,\Delta_{\zz_d})$ coincides.

\begin{ex}\label{ex:triangle}
{\rm
The transpose of the simplicial set map $\Delta[1]\to D(\Delta_{\ZZ_2})$ discussed in Example \ref{ex:edge-dist} under the adjunction in (\ref{eq:adjunction-Dec})  is given by 
$$
p: \Delta[2] \cong \Delta[0] \ast \Delta[1] \to D(N\ZZ_2).
$$
Let $\iota_2$ denote the unique nondegenerate $2$-simplex of $\Delta[2]$. Let $x,y,z$ denote the $1$-simplices of $\Delta[2]$ given by $d_2\iota_2$, $d_0\iota_2$ and $d_1\iota_2$; respectively. Then $p$ can be represented as follows:
$$ 
\includegraphics[width=.2\linewidth]{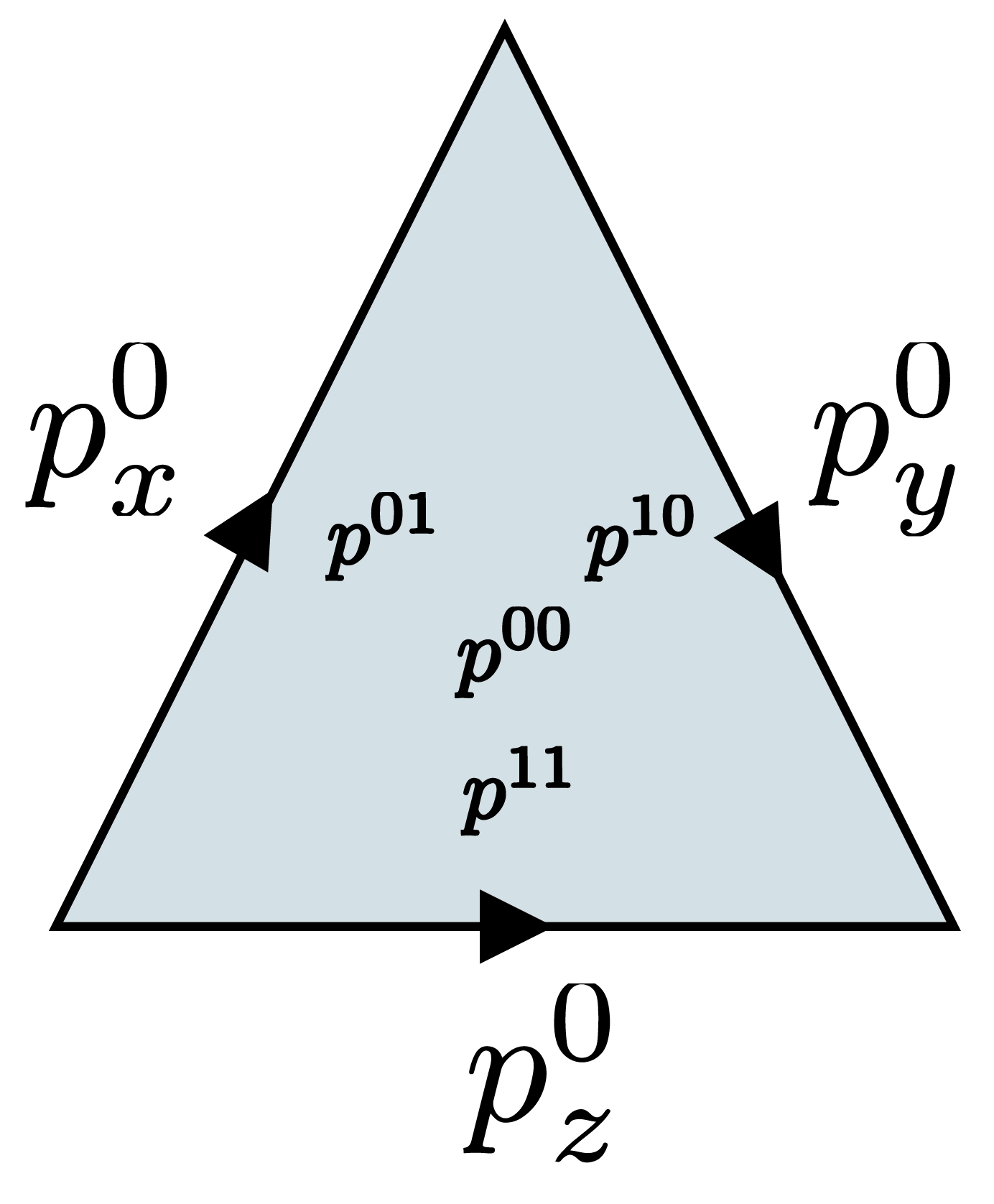}  
$$
The distribution $p_z=d_1 p_{\iota_2}$ is given by
$$
p_z^0 = p^{00}+p^{11},
$$
and $p_x^0$, $p_y^0$ are as before; see Equation (\ref{eq:x-y}).
We will write $x\oplus y$ for the $z$ edge, the XOR of $x$ and $y$. For more details on this interpretation see \cite{okay2022simplicial}.
}
\end{ex}

\begin{ex}\label{Ex:Sqqq} 
{\rm
In Example \ref{Ex: CHSH} the simplicial complex $\Sigma$ is the boundary of a square. The join $\Delta[0]\ast X_\Sigma$ (or rather the ordered version $\Delta[0]\ast X^\ord_\Sigma$) is the cone of this one-dimensional space, which is the square consisting of four triangles:
$$ 
\includegraphics[width=.26\linewidth]{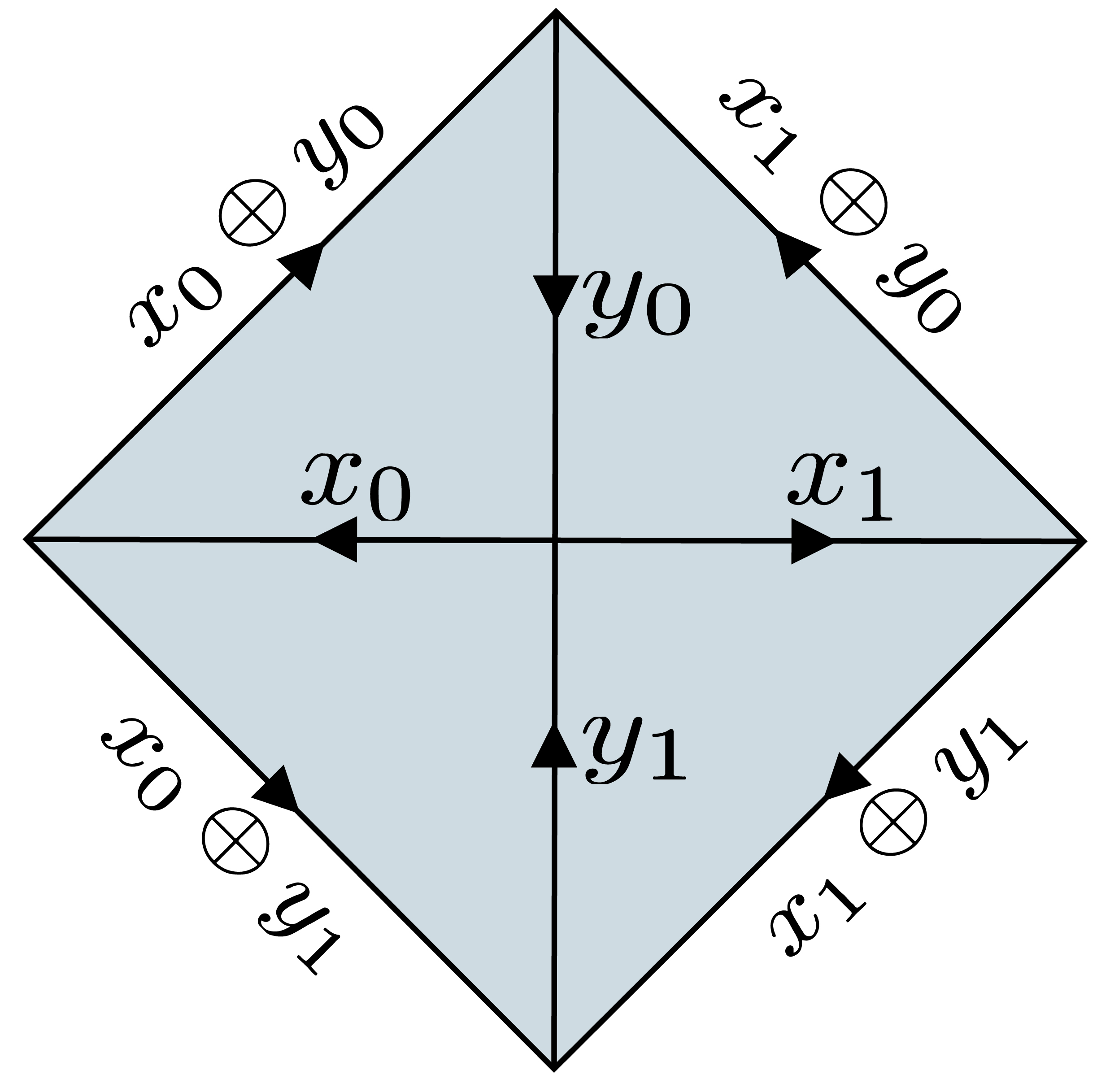}  
$$
This simplicial set is used in the topological proof of Fine's theorem \cite[Section 4.5]{okay2022simplicial}.
{See also \cite{okay2022mermin} for an alternative proof of Fine's theorem that uses a different simplicial realization.}
}
\end{ex}


	\section{Convex categories}\label{sec:ConvexCat}
	
	In this paper by a category we mean a locally small category in the sense that the collection of morphisms between any two objects forms a set. Let $\catCat$ denote the category of (locally small) categories. 
In this section we introduce the notion of convex categories by upgrading the distributions functor $D_R$ to a monad on {$\catCat$}.
	
\subsection{Convex categories} \label{subsec:conv-cat}

Let $X$ and $Y$ be sets. We define a multiplication map 
$m_{X,Y}:D_R(X)\times D_R(Y)
 \to D_R(X\times Y)$, by sending 
 $(p,q)$ to the distribution {$p\cdot q$} given by
\begin{equation}\label{Mulmap}
(p \cdot q)(x,y)=p(x)q(y),\;\; x\in X,\;y\in Y.
\end{equation}
%
This is a well-defined distribution since 
$$
		\begin{aligned}
			\sum_{(x,y) \in X \times Y} (p \cdot q)(x,y) =\sum_{(x,y) \in X \times Y}
			p(x)q(y)
			=\sum_{x \in X} p(x)\sum_{y\in Y} q(y)
			=1		\end{aligned}
		$$
and
the support of $p\cdot q$ is 
a subset of the finite set
		$\set{(x,y):\, p(x) \neq 0,\; q(y)\neq 0}$.

\begin{lem}\label{SecofPro}
%
The map $m_{X,Y}$ is a section of  $D_R(r_1)\times D_R(r_2):D_R(X\times Y)\to D_R(X)\times D_R(Y)$, where 
$r_1: X \times Y \to 
X$ and $r_2: X \times Y \to Y$ are the canonical projection maps.
%
\end{lem}

\begin{proof}
Given $(p,q) \in D_R(X)\times D_R(Y)$,  
%
%
we have
$$
(D_R(r_1)(p\cdot q))(x)= 
\sum_{y \in Y} (p\cdot q)(x,y)=
\sum_{y \in Y}p(x)q(y)= p(x).
$$
Similarly $D_R(r_2)(p\cdot q)=q$.
\end{proof}

Let $\catC$ be a category. For objects $X,Y,Z$ of $\catC$ we define a  map
	\begin{equation}\label{Convul}
		\ast: D_R(\catC(X,Y)) \times D_R(\catC(Y,Z)) \to D_R(\catC(X,Z)),\;\;{(p,q)\mapsto q\ast p},
	\end{equation}
to be the composite of 
$m_{C(X,Y),C(Y,Z)}$ with $D_R(\circ)$, 
where $\circ$ is the composition in $\catC$. 
So we have	 
	\begin{equation}\label{eq:composition-DC}
(q \ast p)(f)=\sum_{\,g_2 \circ g_1=f}q(g_2)p(g_1),\;\;\; \forall   f \in \catC(X,Z).
	\end{equation}
	where the sum runs over morphisms $g_1:X\to Y$ and $g_2:Y\to Z$ such that $g_2\circ g_1=f$ in $\catC$.
 
	\begin{lemma}\label{m1m2}
		Given $X,Y,Z$, objects in a category $\catC$, and given morphisms 
		$p=\sum_{f \in \catC(X,Y)}\alpha_f\delta^f \in D_R(\catC(X,Y))$ and 
		$q=\sum_{g \in \catC(Y,Z)}\beta_g\delta^g \in D_R(\catC(Y,Z))$ we have
		$$
		\left(\sum_{g \in \catC(Y,Z)}\beta_g \delta^g \right) \ast
		\left(\sum_{f \in \catC(X,Y)}\alpha_f\delta^f \right)=
		\sum_{f \in \catC(X,Y),\, g \in \catC(Y,Z)}\beta_{g}\alpha_{f}
		\,\delta^{g \circ  f}.
		$$
		In particular,
		$
		\delta^g \ast \delta^f =\delta^{g \circ f}
		$.
	\end{lemma}
	\begin{proof}
		Given $h \in \catC(X,Z)$ we have
		$$
		\begin{aligned}
			\left(\sum_{g \in \catC(Y,Z)}\beta_g \delta^g \right) \ast
			\left(\sum_{f \in \catC(X,Y)}\alpha_f\delta^f \right)(h) 
			&=
			\sum_{h_2 \circ h_1 =h}\left(\sum_{g \in
				\catC(Y,Z)}\beta_g\delta^g\right)(h_2) \cdot \left(\sum_{f \in
				\catC(X,Y)}\alpha_f\delta^f\right)(h_1)\\
			&=
			\sum_{h_2\circ h_1 =h}\beta_{h_2} \cdot
			\alpha_{h_1}\\
			&=
			\sum_{h_1 \in \catC(X,Y),\,h_2 \in
				\catC(Y,Z)}\beta_{h_2}\alpha_{h_1}\delta^{h_2\circ h_1}(h).
		\end{aligned}
		$$
		
	\end{proof}
	\begin{lem}\label{lem:def-DC}
		Let $\catC$ be a category. Consider the data consisting of
		\begin{itemize}
			\item the collection of objects of $\catC$,
			\item for objects $X,Y$ of $\catC$ the set $D_R(\catC(X,Y))$ of morphisms.
		\end{itemize}
		This data together with Equation (\ref{eq:composition-DC}) as the composition map specifies a category $D_R(\catC)$. 
	\end{lem}
	\begin{proof}
		The identity morphism in $D_R(\catC)(X,X)$ is given by the delta distribution $\delta^{\Id_X}$.
		Given $p \in D_R(\catC)(X,Y)$ and $f \in \catC(X,Y)$ we have 
		$$
		(\delta^{\Id_Y }\ast p)(f)=\sum_{g_2 \circ g_1=f}\delta^{\Id_Y }(g_2)p(g_1)
		=\delta^{\Id_Y }(\Id_Y)p(f)=p(f).
		$$
		Similarly, we have $p \ast \delta^{\Id_X}=p$.
		For associativity we check that for 
		$s \in D_R(\catC)(X,Y)$, $q \in D_R(\catC)(Y,Z)$ and $p \in D_R(\catC)(Z,W)$ we have
		$$
		\begin{aligned}
			((p \ast q) \ast s)(f) &=
			\sum_{g_2 \circ g_1=f}(p\ast q)(g_2)s(g_1)\\
			&=\sum_{g_2 \circ g_1=f}\left(\sum_{h_2 \circ
				h_1=g_2}p(h_2)q(h_1)\right)s(g_1)\\
			&=\sum_{g_2 \circ g_1=f}\sum_{h_2 \circ
				h_1=g_2}p(h_2)q(h_1)s(g_1) \\
			&=\sum_{h_2 \circ h_1 \circ g_1=f}p(h_2)q(h_1)s(g_1).
		\end{aligned}
		$$
		A similar calculation shows that $(p \ast (q \ast s))(f)$ gives the same result.
	\end{proof}

	\begin{lem}\label{lem:def-DF}
		Given a functor $F : \catC_1 \to \catC_2$ in $\catCat$ let us define the following maps:
		\begin{itemize}
			\item $D_R(F):\ob(D_R(\catC_1))\to \ob(D_R(\catC_2))$ defined  as the map $F$ between the objects of the categories.
			\item $D_R(F)_{X,Y}:D_R(\catC_1)(X,Y)\to D_R(\catC_2)(F X,F Y)$ defined by applying $D_R$ functor to the map $F_{X,Y}$ between the morphism sets.
		\end{itemize}
		Then these maps specify a functor $D_R(F):D_R(\catC_1)\to D_R(\catC_2)$. 
	\end{lem}
	\begin{proof}
		Given $X \in \ob(D_R(\catC_1))$, we have
		$$
		D_R(F)(\delta^{\Id_X})(f)=\sum_{F(g)=f}\delta^{\Id_X}(g)=
		\begin{cases}
			1   &  f=\Id_{F(X)}     \\  
			0 &  \text{otherwise.}
		\end{cases}
		$$
		Therefore $D_R(F)(\delta^{\Id_X})=\delta^{\Id_{F(X)}}$.
		Now given $p \in D_R(\catC_1)(Y,Z)$, $q \in D_R(\catC_1)(X,Y)$ and $f \in \catC_2(FX,FZ)$, then
		$$
\begin{aligned}		
D_R(F)(p \ast q)(f) &=\sum_{F(g)=f}(p\ast q)(g) \\
&= 	\sum_{F(g)=f} \sum_{{f_2 \circ
				f_1=g}}p(f_2)\cdot q(f_1) \\
& = \sum_{F(f_2 \circ
			f_1)=f}p(f_2)\cdot q(f_1) \\
&=\sum_{F(f_2)\circ F(f_1)=f }p(f_2)\cdot q(f_1)\\
&=	\sum_{{h_2 \circ h_1=f}}
			\sum_{F(f_2)=h_2,\, F(f_1)=h_1 }p(f_2)\cdot q(f_1)\\
&=\sum_{{h_2 \circ 	h_1=f}}\left(\sum_{{F(f_2)=h_2}}p(f_2)\right)\cdot \left(\sum_{{F(f_1)=h_1}}
			q(f_1)\right) \\
&= \sum_{{h_2 \circ h_1 =f}}D_R(F)(p)(h_2)\cdot
			D_R(F)(q)(h_1)\\
&= (D_R(F)(p) \ast D_R(F)(q))(f).			
\end{aligned}		
		$$
	\end{proof}

{Combining Lemma \ref{lem:def-DC} and Lemma \ref{lem:def-DF} we can upgrade $D_R$ to a functor on $\catCat$.
}

	\begin{cor}\label{DRFun}
		Sending $\catC \mapsto D_R(\catC)$ and $F\mapsto D_R(F)$ gives a functor
		$D_R: \catCat \to \catCat$.
	\end{cor}
	\begin{proof}
		Follows from the 
		the functoriality of $D_R : \catSet \to \catSet$. 
	\end{proof}
	%
	%
	
{Next we will show that $D_R$ is a monad. For this we introduce the structure morphisms.}	
	
	\begin{lem}\label{DeltaMuNat}
		Given a category $\catC$ and objects $X,Y$ of $\catC$ we define the following maps:
		\begin{itemize}
			\item $\delta_\catC: \catC \to D_R(\catC)$ to be the identity map on the objects and 
			$\delta_\catC :\catC(X,Y) \to D_R(\catC)(X,Y)$ to be $\delta_{\catC(X,Y)}$.
			\item $\mu_\catC : D_R(D_R(\catC)) \to D_R(\catC)$ to be the identity on the objects and       
			$\mu_\catC :D_R(D_R(\catC))(X,Y) \to D_R(\catC)(X,Y)$ to be $\mu_{\catC(X,Y)}$.
		\end{itemize}
		These maps specify the following natural transformations:
		$$
		\delta: \idy_\Cat \to D_R\;\;\text{ and }\;\; \mu:D_RD_R \to D_R.
		$$
		%
	\end{lem}
	\begin{proof}
		First, we prove that $\delta_\catC$ and $\mu_\catC$ are functors. 
		Given $X \in \ob(\catC)$,  we have 
		$\delta_\catC(\Id_X)=\delta^{\Id_X}$, which is the identity morphism in $D_R(\catC)(X,X)$.
		Similarly, 
 $\mu_\catC(\delta^{\delta^{\Id_{X}}})=
 \delta^{\Id_X}$.
		Given $f \in \catC(X,Y)$, $g \in \catC(Y,Z)$ we have
		$$
		\delta_\catC(g \circ f)=\delta^{g \circ f} = \delta^{g} \ast \delta^{f}
		=\delta_\catC(g) \ast \delta_\catC(f),
		$$ 
where the second equality follows from Lemma \ref{m1m2}.		For $Q_1 \in D_R(D_R(\catC(X,Y)))$, $Q_2 \in D_R(D_R(\catC(Y,Z))$ and $f \in \catC(X,Z)$ we calculate
		$$
		\begin{aligned}
			\mu_\catC(Q_2 \ast Q_1)(f)&=\sum_{q \in
				D_R(\catC(X,Z))}(Q_2 \ast Q_1)(q)q(f)=\sum_{q \in
				D_R(\catC(X,Z))}\left(\sum_{p_1 \ast p_2 = q}Q_2(p_2) \cdot Q_1(p_1)\right)q(f)\\
			&=\sum_{q \in D_R(\catC(X,Z))}\sum_{p_2 \ast p_1 = q}Q_2(p_2) \cdot Q_1(p_1)(p_2\ast p_1)(f) \\
			&=\sum_{q \in D_R(\catC(X,Z))}\sum_{p_2 \ast p_1 = q}Q_2(p_2) \cdot Q_1(p_1) \sum_{g_2 \circ g_1 = f}p_2(g_2) \cdot p_1(g_1) \\
			&=\sum_{q \in D_R(\catC(X,Z))}\sum_{p_2 \ast p_1 = q}\sum_{g_2 \circ g_1 = f}Q_2(p_2) \cdot Q_1(p_1)
			\cdot p_2(g_2) \cdot p_1(g_1)\\
			&=\sum_{p_1 , p_2}\sum_{g_2 \circ g_1 = f}Q_2(p_2) \cdot Q_1(p_1)\cdot
			p_2(g_2) \cdot p_1(g_1)\\
			&=\sum_{g_2\circ g_1=f  }\sum_{p_2 \in
				D_R(\catC(Y,Z))}Q_2(p_2)\cdot p_2(g_2)\sum_{p_1 \in D_R(C(X,Y))}Q_1(p_1)\cdot p_1(g_1)\\
				&= \sum_{g_2\circ g_1=f}\mu_\catC(Q_2)(g_2) \mu_\catC(Q_1)(g_1) \\
			&=	(\mu_\catC(Q_2)  \ast \mu_\catC(Q_1))(f).
		\end{aligned}
		$$
		The naturality of $\delta$ and $\mu$ follows from the	naturality of the set-theoretic versions of $\delta$ and $\mu$.
	\end{proof}
	
	Combining Corollary \ref{DRFun} and Lemma \ref{DeltaMuNat} we obtain the main 
	result of this subsection.
	
	\begin{cor}\label{DRCat}
		The triple $(D_R,\delta,\mu)$ is a monad on $\catCat$
	\end{cor}

	\begin{defn}\label{ConCatttt}
{\rm	
A $D_R$-algebra in $\catCat$ will be called an
{\it $R$-convex category}. 
		We will write $\catConvCat_R$ for the category of $R$-convex categories.
}		
	\end{defn}
%

{This notion of  convex category is different than the one introduced in \cite[Definition 5.1]{Jacobs_2011}.
For example, $(\rr,\cdot)$ is an $\rr_{\geq 0}$-convex category with one object, but not a convex category in the latter sense since coproducts do not exist in this category.
Next, we provide an explicit criterion to check whether a category is $R$-convex.
}

\begin{prop}	\label{ConvCat2}
	Let $\catC$ be a category such that the set $\catC(X,Y)$ of morphisms is an $R$-convex set for all objects $X,Y$ of $\catC$.
We define the following maps:
\begin{itemize}
\item $\ob(D_R(\catC)) \to \ob(\catC)$ to be the identity map.
\item $D_R(\catC)(X,Y) \to \catC(X,Y)$ to be the structure map $\pi^{\catC(X,Y)}$. 
\end{itemize}
These maps give a well defined functor 
$\pi^\catC : D_R(\catC) \to \catC$ if and only if 
$(\catC,\pi^\catC)$ is an  $R$-convex category.
%
%
%
\end{prop}

		
{Categories enriched over $R$-convex sets are examples of convex categories.}		

\begin{prop}\label{EnrichSOconv}
Any category $\catC$ enriched over 
$\catConv_R$ is an $R$-convex category.
\end{prop}
\begin{proof}
The set $\catC(X,Y)$ is $R$-convex for every object $X,Y$.
We will prove that the maps in Proposition \ref{ConvCat2} give a well-defined functor
$\pi^{\catC}:D_R(\catC) \to \catC$.
We have 
$\pi^{\catC(X,X)}(\delta^{\Id_X})
=\pi^{\catC(X,X)}(\delta_{\catC(X,X)}(\Id_X))
=\Id_X$.
Given $X,Y,Z \in \ob(\catC)$, let 
$r_1: \catC(X,Y)\times \catC(Y,Z) 
\to \catC(X,Y)$ and $r_2: \catC(X,Y)\times \catC(Y,Z) 
\to \catC(Y,Z)$ be the projection maps.
Using the proof of \cite[Theorem  5.6.5]{riehl2017category},
we obtain that
\begin{equation}\label{PiProd} 
\pi^{\catC(X,Y) \times \catC(Y,Z)}=(\pi^{\catC(X,Y)} \times \pi^{\catC(Y,Z)}) \circ (D_R(r_1) \times D_R(r_2) ) 
\end{equation}
By composing Equation (\ref{PiProd}) with 
$m_{\catC(X,Y),\catC(Y,Z)}$, defined in Equation (\ref{Mulmap}), from the right and using  Lemma \ref{SecofPro} we obtain that   
$(\pi^{\catC(X,Y) \times \catC(Y,Z)})
\circ m_{\catC(X,Y),\catC(Y,Z)} =\pi^{\catC(X,Y)} \times \pi^{\catC(Y,Z)}
$. This gives the following commutative
diagram:
$$
\begin{tikzcd}
D_R(\catC(X,Y))\times D_R(\catC(Y,Z)) \arrow[rr] \arrow[rd,"m"]
 \arrow[dd,"\pi\times \pi"]
 && \catC(X,Z) \arrow[dd,"\pi"]  \\
 & D_R(\catC(X,Y)\times\catC(Y,Z)) 
 \arrow[ru,"D_R(\circ)"]
 \arrow[ld,"\pi"]  & \\
\catC(X,Y)\times \catC(Y,Z) 
\arrow[rr,"\circ"] && \catC(X,Z)
\end{tikzcd}
$$
By Proposition \ref{ConvCat2} we conclude that $\catC$ is an $R$-convex category. 
\end{proof}

\begin{rem}\label{rem:enriched-convex}
{\rm
The notion of $R$-convex category is   weaker than the notion of category enriched over 
$\catConv_R$. For example, 
the monoid $(\rr, \cdot)$ as a category with one object is an $\rr_{\geq 0}$-convex category, but the product is not a morphism in 
$\catConv$. Therefore it is not enriched over convex sets.
{Similarly, the Kleisli category $\catSet_{D_R}$ is not enriched over $R$-convex sets, but it is an $R$-convex category. In Proposition \ref{pro:Kleisli-conv-cat} we will show that its simplicial version $\catsSet_{D_R}$ is also an $R$-convex category, which is again not enriched over $R$-convex sets.
}
}
\end{rem}

\subsubsection{Real convex categories}

{The isomorphism between the categories of   $\RR_{\geq 0}$-convex sets and real convex sets (Proposition \ref{Conv=ConvR})   extends to the case of $\RR_{\geq 0}$-convex categories. We have the following equivalent characterizations of $\RR_{\geq 0}$-convex categories.}

\begin{prop}\label{ThConvCat} 
Let $\catC$ be a category such that the set $\catC(X,Y)$ of morphisms is a real convex set for all objects $X,Y$ of $\catC$.
Let $\pi^\catC : D(\catC) \to \catC$ be defined as in 
Proposition \ref{ConvCat2} where $\pi^{\catC(X,Y)}$ is the corresponding structure map of $\catC(X,Y)$ when regarded 
as an $\RR_{\geq 0}$-convex set (see Proposition \ref{Conv=ConvR}). 
%
Then the following statements are equivalent.
		\begin{enumerate}
			\item $(\catC,\pi^\catC)$ is an  $\rr_{\geq 0}$-convex category.
			
			\item For objects $X,Y,Z$ of $\catC$, morphisms $f_i \in \catC(X,Y)$, $g_i \in \catC(Y,Z)$, and numbers $\alpha_i ,\beta_i \in [0,1]$, where $1 \leq i \leq k$,  satisfying
			$\sum_{i} \alpha_i = \sum_{i} \beta_i = 1$ we have 
			$$
			\left(\sum_{i} \beta_i  g_i\right)\circ \left(\sum_{i} \alpha_i f_i\right) 
			=\sum_{i,j} \beta_j\alpha_i\, g_j \circ f_i.
			$$

			\item For objects $X,Y,Z$ of $\catC$, morphisms $f_1,f_2 \in \catC(X,Y)$, $g_1,g_2 \in \catC(Y,Z)$, and numbers 
			$\alpha ,\beta \in [0,1]$ we have the following equality:
$$
\begin{aligned}
(\beta g_1 + (1-\beta) g_2) & \circ (\alpha f_1 + (1-\alpha) f_2) 
 \\
&=\beta \alpha  g_1 \circ f_1 + \beta (1-\alpha) g_1\circ  f_2 +
(1-\beta) \alpha g_2 \circ f_1
 + (1-\beta) (1-\alpha) g_2 \circ f_2.
\end{aligned}
$$
\end{enumerate} 
\end{prop}
\begin{proof}
We begin by showing that $(1)$ implies $(2)$.	
Let $X,Y,Z \in \ob(\catC)$, $f_i \in \catC(X,Y)$, $g_i \in \catC(Y,Z)$, $\alpha_i ,\beta_i \in [0,1]$, where $1 \leq i \leq k$ and
$\sum_{i} \alpha_i = \sum_{i} \beta_i = 1$.
Using Lemma \ref{m1m2} and the functoriality of $\pi^\catC$ we obtain 
$$
\begin{aligned}
			\sum_{i,j} \beta_j\alpha_i\, g_j \circ f_i &=
			\pi^{\catC(X,Z)}\left(\sum_{i,j}\beta_{j}\alpha_{i}
			\delta^{g_j \circ  f_i}\right)\\
			&=
			\pi^{\catC(X,Z)}\left((\sum_{j}\beta_{j} \delta^{g_j}) \ast
			(\sum_{i}\alpha_{i}\delta^{f_i})\right)\\
			&=\pi^{\catC(Y,Z)}\left(\sum_{j}\beta_{j} \delta^{g_j}\right) 
\circ
			\pi^{\catC(X,Y)}\left(\sum_{i}\alpha_{i} \delta^{f_i}\right)\\
			&=\left(\sum_{j} \beta_j  g_j\right)\circ \left(\sum_{i} \alpha_i f_i\right) .
		\end{aligned}
		$$

		Conversely, assume that $(2)$ holds. By 
Proposition \ref{ConvCat2} it is enough to prove that $\pi^C$ is a functor. Given $p \in D_{\RR_{\geq 0}}(\catC(X,Y))$, $q \in D_{\RR_{\geq 0}}(\catC(Y,Z))$ we have
$$
\begin{aligned}
\pi^{\catC(X,Z)}(q \ast p) &= \sum_{f \in \catC(X,Z)} (q \ast p)(f)f \\
& = \sum_{f \in \catC(X,Z)}  \left(\sum_{g_1 \circ g_1 =f} q(g_2) p(g_1) \right) f \\ 
			&=\sum_{g_1 \in C(X,Y),\, g_2 \in \catC(Y,Z)} q(g_2)p(g_1) g_2\circ g_1\\
&= \left(\sum_{g_2 \in  \catC(Y,Z)} q(g_2)g_2 \right)
			\circ \left(\sum_{g_1 \in  \catC(X,Y)} p(g_1)g_1 \right)   \\
&=			\pi^{\catC(Y,Z)}(q)\circ \pi^{\catC(X,Y)}(p)
\end{aligned}
$$
Finally, note that $(3)$ is a special case of $(2)$ and one can get $(2)$ from $(3)$ by induction.
	\end{proof}


%
	We will write  $\catConvCat$ for the category $\catConvCat_{\rr_{\geq 0}}$ 
	and refer to the objects of this category as real convex categories. 
Note that Proposition \ref{ThConvCat}
implies that 
 the category of $\RR$-vector spaces with the canonical convex structure on the $\RR$-vector space of linear maps is a real convex category.

	\subsection{Kleisli category as a convex category} \label{sec:kleisli-convex-cat}
	Recall  that the Kleisli category $\catC_T$ of a monad $T:\catC\to \catC$ has the same objects as $\catC$ and for objects $X,Y$ its morphisms are given by $\catC(X,T Y)$. For an object $X$ the identity morphism in $\catC_T(X,X)$ is given by  $\delta_X:X\to TX$. The composition of two morphisms $f:X\to TY$ and $g:Y\to TZ$ is defined by
	$$
	g \diamond f: X \xrightarrow{f} TY \xrightarrow{T(g)} T(TZ) \xrightarrow{\mu_Z} TZ. 
	$$
	We are interested in the Kleisli category of the distribution monad $D_R:\catsSet \to \catsSet$.
	In this case the composition can be explicitly written as
	$$
(q \diamond p)_n(x)= \sum_{y\in Y_n}p_n(x)(y)q_n(y),  
	$$ 
{where $p\in \catsSet(X,D_R(Y))$, $q\in \catsSet(Y,D_R(Z))$ and $x\in X_n$.} 	
%
	%
	
	%
	
	
	%
	%

	\begin{prop}\label{pro:Kleisli-conv-cat}
		The Kleisli category $\catsSet_{D_R}$ is an $R$-convex 
		category.  
	\end{prop}
	\begin{proof}
		The map $\pi :D_R(\catsSet_{D_R}) \to \catsSet_{D_R}$ which employs $\catsSet_{D_R}$ with the structure of an $R$-convex category is defined to be identity on the collection of objects. For simplicial sets $X,Y$ the map 
		$$
		\pi_{X,Y} : D_R(\St(X,D_R(Y))) \to \St(X,D_R(Y)) 
		$$
is defined to be $\pi^{\catsSet(X,D_R(Y))}$
 (see {Equation (\ref{eq:Rconv-structure-sSet})} and Proposition \ref{Hommm}).
We need to show, as a consequence Proposition \ref{ConvCat2}, that $\pi$ is a functor. That is, for all $P \in D_R(\catsSet(X,D_RY))$ and $Q\in D_R(\catsSet(Y,D_RZ))$ we have $\pi(Q \ast P)=\pi(Q) \diamond \pi(P)$.
		This follows from the following computation: {For $x\in X_n$ we have
		$$
		\begin{aligned}
			\pi(Q \ast P)_n(x) &=\sum_{p \in  \catsSet(X,D_R(Z)) } (Q\ast P)(p)p_n(x) \\
			&=\sum_{p \in \catsSet(X,D_R(Z)) } \sum_{q' \diamond q =p}
			Q(q') P(q ) p_n(x)\\
			&=\sum_{q \in \catsSet(X,D_R(Y)),\, q' \in\catsSet(Y,D_R(Z)) } 
			Q(q') P(q) (q'\diamond q)_n (x)\\
			&=\sum_{q,q'}P(q)Q(q')\sum_{y\in Y_n}q_n(x)(y)q'_n(y)\\
			&=\sum_{y \in Y_n} \sum_{q \in \St(X,D_R(Y))} P(q)q_n(x)(y) 
			\sum_{q' \in \St(Y,D_R(Z))}Q(q')q'_n(y)\\
			&=\sum_{y\in Y_n}\pi(P)_n(x)(y)\pi(Q)_n(y)\\
			&= (\pi(Q) \diamond \pi(P))_n(x) .
		\end{aligned}
		$$}
	\end{proof}
	%

	Let  $F_{T} : \catC \to\catC_{T}$ denote the functor defined as follows:
\begin{itemize}
\item $\ob(\catC) \to \ob(\catC_T)$ is the identity functor.
\item $F_T :\catC(X,Y) \to 
\catC_T(X,Y)$ is defined to be $(\delta_Y)_\ast$, that is  $f: X \to Y$  is sent to  the composite $F_T(f):X \xrightarrow{f} Y \xrightarrow{\delta_Y} TY$.
\end{itemize}	
The functor $F_T$ 
	has a right adjoint $U_T:\catC_T \to \catC$ 
	(see \cite[Lemma 5.2.11]{riehl2017category}). 
	In fact, the monad $T$  arises from this adjunction. 
	{In the following we will consider the free convex category $D_R(\catsSet)$, and
the Kleisli category $\catsSet_{D_R}$, which is also a convex category (Proposition \ref{pro:Kleisli-conv-cat}).
The $\Theta$-map defined in Equation (\ref{eq:Theta}) can be given a categorical interpretation using the theory of convex categories.
}

	\begin{prop}
		The transpose of the functor $F_{D_R}: \catsSet \to \catsSet_{D_R}$
		with respect to the 
		adjunction $D_R: \catCat \adjoint \catConvCat_R: U$
		is the functor 
		$\Theta:  D_R(\St) \to \St_{D_R}$ which is defined as identity on the objects and as the map $\Theta_{X,Y}$ 
		on morphisms. 
	\end{prop}
	\begin{proof}
		Using Proposition \ref{ThetaUniqqq} we see that 
		$\Theta = \pi^{\St_{D_R}}\circ D_R(F_{D_R})$. Then
	Remark \ref{AdjFormmm} implies that  $\Theta$ is 	the corresponding transpose of $F_{D_R}$.
	\end{proof}
	%

	\section{Convex monoids}\label{Sec: WI}
	%
	%
	%
	
	
 
		In Section \ref{sec:ConvexCat} we introduced the notion of a convex category. Now, in this section we will specialize to convex monoids, that is convex categories with a single object. {We will introduce a  weak notion of invertibility for convex monoids. 
This definition is inspired by the definition of noncontextuality for simplicial distributions. Later in Section \ref{Sec: SimDis} we will see that the two notions coincide for cases of interest.		
For real convex monoids we introduce the notion of invertible fraction to quantify the closeness of an element to being invertible. 
}

		The monad $D_R$ acting on $\catCat$ restricts to a monad on the category $\catMon$ of monoids.
		A $D_R$-algebra $(M,\pi^M)$ over this monad is called an {\it $R$-convex monoid}.
		By Proposition \ref{ConvCat2} this is equivalent to saying that $(M,\pi^M)$ is an $R$-convex set and the map $\pi^M:D_R(M)\to M$ is a homomorphism of monoids.
		We will write $\catConvMon_R$ for the category of $R$-convex monoids. Proposition \ref{ThConvCat} specializes to give the following result.

	%
	
	%
	\begin{cor}\label{ThrProp}
		Let $(M,\cdot)$ be a monoid which is also a real convex set.
		Let $\pi^M$ denote 
		the structure map of $M$ when regarded as an $\RR_{\geq 0}$-convex set.  
		Then the following statements are equivalent.
\begin{enumerate}
\item $(M,\pi^M)$ is an $\rr_{\geq 0}$-convex monoid.
			
\item For $m_i,n_i \in M$ and 
			$\alpha_i , \beta_i\in [0,1]$, where  $1 \leq i \leq k$
			and $\sum_{i} \alpha_i =\sum_{i} \beta_i = 1$,  
			we have 
			$$
\left(\sum_{i} \alpha_i  m_i\right)\cdot \left(\sum_{i} \beta_i n_i\right) 
			=\sum_{i,j} \alpha_i \beta_j\, m_i \cdot n_j
			$$

			\item For $m_1,m_2,n_1,n_2 \in M$ and 
			$\alpha , \beta \in [0,1]$  we have the following equality 
$$
\begin{aligned}
(\alpha m_1 + (1-\alpha) m_2) &\cdot (\beta n_1 + (1-\beta) n_2)\\
&=\alpha \beta m_1 \cdot n_1 + \alpha (1-\beta) m_1 \cdot n_2 
+ (1-\alpha) \beta m_2 \cdot n_1  
+ (1-\alpha) (1-\beta) m_2 \cdot n_2.
\end{aligned}
$$

		\end{enumerate} 
	\end{cor}

	\begin{remark}{\rm
			By part (3) of Corollary \ref{ThrProp} one can see that the notion of a convex monoid given in \cite[Definition 9]{Roumen_2016}
			is a special case of $\rr_{\geq 0}$-convex monoid.
		}
	\end{remark}
	As in the case of categories, from now on we will write 
	$\catConvMon$ for $\catConvMon_{\rr_{\geq 0}}$ and refer to the objects of this category as real convex monoids.

	
	
	\begin{example}\label{Ex1 : CM}
		{\rm
			A semiring $R$ can be given the structure of an $R$-convex set by defining 
			$
			\pi^{R}: D_{R}(R) \to R
			$
			as follows:
			$$
			\pi^R(p)=\sum_{x \in R} p(x)x .
			$$
			{Since} $R$ is commutative $\pi^R$ is a homomorphism of monoids, hence $R$ becomes an $R$-convex monoid.
		}
		
		
		%
		%
		%
	\end{example}

	\begin{example}\label{Ex2: CM}
		{\rm
			The set $\rr$  of real numbers is a convex set, and 
with the following product 
			$$
			x_1 \vartriangle x_2 := x_1 x_2 + (1-x_1)(1-x_2),\;\;\;\; x_1,x_2\in \RR,
			$$ 
			is a monoid with $1$ as the identity.
			In addition, part  $(3)$ of {Corollary} \ref{ThrProp} holds, hence $(\rr,\pi^\rr)$ is a 
			real convex monoid.  
			Note that $([0,1],\pi^\rr,\vartriangle)$ is  a real subconvex monoid of $(\rr,\pi^\rr,\vartriangle)$.
		}
	\end{example}
	%

	%

\begin{example}{\rm
The set of continuous functions from $\rr$ to $\rr$   is a real 
convex set. With the composition operation it is a monoid, but not a real convex monoid.  
}
\end{example}

	\subsection{Weak invertibility} \label{sec:weak-invertibility}
	%
In this section, we will introduce the notion of weak invertibility for monoids. We begin by relating monoids to groups. For a monoid $M$ let $I(M)$ denote the subset of invertible elements. This construction defines a functor $I:\catMon \to \catGrp$, which turns out to be the right adjoint of the inclusion functor $j:\catGrp \to \Mon$; see \cite[Example 2.1.3 (d)]{leinster2014basic}. The composition of the two adjunctions 
$j:\catGrp \adjoint \catMon : I$ and $ 
D_R:\catMon \adjoint \catConvMon_R:U$
gives us 
the following adjunction
		\begin{equation}\label{eq:adj-I}
			D_R:\catGrp \adjoint \catConvMon_R: I.
		\end{equation}
For simplicity, we will use the notation $M^\ast$
instead of $I(M)$.

Let $i_M:M^\ast \to M$ denote the inclusion map. 
We will consider   $D_R(M^*)$ as a subset of $D_R(M)$ via the map $D_R(i_M)$.
The restriction of $\pi^M: D_R(M)\to M$ to  $D_R(M^*)$ will be denoted by $\tilde \pi^M$.


	\begin{defn}\label{weakk}
		{\rm
			An element $m \in M$ is called 
			{\it weakly invertible} if it lies in the image of  $\tilde\pi^M:D_R(M^*)\to M$.
			
		}
	\end{defn}
	%
	In particular, every invertible element is weakly invertible. 
	
	\begin{example}\label{Ex2: W.I}
		{\rm
			In Example \ref{Ex1 : CM} if we let $R=\rr_{\geq 0}$ then $0$ is not weakly invertible since for 
			$p\in D((\rr_{\geq 0})^\ast)$ we have
			$$
			\tilde{\pi}^{\rr_{\geq 0}}(p)= \sum_{x \in \rr_{\geq 0}}p(x)x >0.
			$$
			%
		}
	\end{example}
	
	\begin{example}\label{Ex1: W.I}
		{\rm
			The invertible elements of the subconvex 
			monoid $([0,1],\pi^\rr,\vartriangle)$ (see Example \ref{Ex2: CM}) are $0$ and $1$. For an element $x \in [0,1]$  we have the distribution $p=x \delta_1 + (1-x) \delta_0 \in D([0,1]^\ast)$ that satisfies $\pi^{\rr}(p)=x$.
			Therefore every element of $[0,1]$ is weakly invertible.
		}
	\end{example}
	\begin{prop}\label{mthennm}
Let $n\in M^{\ast}$ and $m \in M$. Then $m$ is weakly invertible if and only if $n\cdot m$ is weakly invertible.
	\end{prop}

\begin{proof}
Suppose that $p \in D_R(M^\ast)$ with $\tilde\pi^M(p)=m$. 
Then we have
\begin{equation}\label{nmWeak}
{\tilde\pi}^M(\delta^n \ast p)={\tilde\pi}^M(\delta^n) \cdot {\tilde\pi}^M(p)=n \cdot m
\end{equation}
In Equation (\ref{nmWeak}) we used the fact that ${\tilde\pi}^M(\delta^n) =\pi^M(D_R(i_M)(\delta^n ))= \pi^M(\delta^n)=n$. This shows that if $m$ is weakly invertible then $n\cdot m$ is weakly invertible. Conversely,
if $n\cdot m$ is weakly invertible then applying this observation to the product $n\cdot m$ (instead of $m$) we obtain that $n^{-1} \cdot (n\cdot m)=m$ is weakly invertible.
\end{proof}

{Next, we provide a criterion for weak invertibility in pull-backs of convex monoids. This is a version for convex monoids of a ``gluing result" for simplicial distributions \cite[Lemma 4.4]{okay2022simplicial}; see also \cite[Section 3.4]{flori2013compositories}.
}

	\begin{lemma}\label{DRX1X2}
Let $R$ be a division {semiring}  which is also zero-sum-free (see \cite{Hutchins_1981}).
	Consider the following diagram in $\catSet$:
		$$
		\begin{tikzcd}
			& X_1 \arrow[d,"f_1"'] \\
			X_2 \arrow[r,"f_2"] & Y
		\end{tikzcd}
		$$
		%
Then the induced map $D_R(X_1\times_Y X_2) \to D_R(X_1)\times_{D_R(Y)}D_R(X_2)$ is surjective.
	\end{lemma}

	\begin{proof}
Given $(p_1,p_2) \in D_R(X_1)\times_{D_R(Y)}D_R(X_2)$, we define 
$
p: X_1 \times_Y X_2 \to R
$
as follows
$$
p(x_1,x_2)= \begin{cases}
\frac{p_1(x_1)p_2(x_2)}{D_R(f_1)(p_1)(f_1(x_1))} &
 \text{if}~ D_R(f_1)(p_1)(f_1(x_1)) \neq 0 \\
0 & \text{otherwise.}
\end{cases}
$$
Note that $D_R(f_1)(p_1)(f_1(x_1)) =D_R(f_2)(p_2)(f_2(x_2)) $
because of the compatibility of $p_1,p_2$ and $x_1,x_2$.
First, we verify that $p \in D_R(X_1 \times_Y X_2)$:
$$
\begin{aligned}
\sum_{(x_1,x_2) \in X_1 \times_Y X_2}p(x_1,x_2)&=
\sum_{(x_1,x_2) \in X_1 \times_Y X_2 \,,\, 
D_R(f_1)(p_1)(f_1(x_1)) \neq 0} 
\frac{p_1(x_1)p_2(x_2)}{D_R(f_1)(p_1)(f_1(x_1))} \\
&=\sum_{x_1 \in X_1 \, , \, 
D_R(f_1)(p_1)(f_1(x_1)) \neq 0 }
 \;\; \sum_{x_2 \in X_2 \, , \, f_2(x_2)=f_1(x_1)}
 \frac{p_1(x_1)p_2(x_2)}{D_R(f_1)(p_1)(f_1(x_1))}\\
&=\sum_{x_1 \in X_1 \, ,\, 
D_R(f_1)(p_1)(f_1(x_1)) \neq 0 }
\frac{p_1(x_1)}{D_R(f_1)(p_1)(f_1(x_1))}
 \;\; \sum_{x_2 \in X_2 \, , \,  f_2(x_2)=f_1(x_1)}
 p_2(x_2)\\
&=\sum_{x_1 \in X_1 \, , \, 
D_R(f_1)(p_1)(f_1(x_1)) \neq 0 }
\frac{p_1(x_1)}{D_R(f_1)(p_1)(f_1(x_1))}
\;  D_R(f_2)(p_2)(f_1(x_1))\\
&=\sum_{x_1 \in X_1   \, , \, 
D_R(f_1)(p_1)(f_1(x_1)) \neq 0 }
p_1(x_1)\\
&=\sum_{x_1 \in X_1}p_1(x_1)=1.
\end{aligned}
$$
In the last line we use the following fact: Since $R$ is a zero-sum-free semiring,  
$D_R(f_1)(p_1)(f_1(x_1)) = 0$ implies $p_1(x_1)=0$. 

Let $r_j :  X_1 \times_Y X_2 \to X_j$, where $j=0,1$, be the projection maps.
We will show that $D_R(r_1)(p)=p_1$. 
Let 
$x\in X_1$. If $D_R(f_1)(p_1)(f_1(x)) \neq 0$, then 
$$
\begin{aligned}
D_R(r_1)(p)(x) &= \sum_{f_2(x_2)=f_1(x)}p(x,x_2)\\
&=
\sum_{f_2(x_2)=f_1(x)}
\frac{p_1(x)p_2(x_2)}{D_R(f_1)(p_1)(f_1(x))}\\
&=\frac{p_1(x)}{D_R(f_1)(p_1)(f_1(x))}
\sum_{f_2(x_2)=f_1(x)} p_2(x_2)\\
&=\frac{p_1(x)}{D_R(f_1)(p_1)(f_1(x))}
D_R(f_2)(p_2)(f_1(x))
=p_1(x).
\end{aligned}
$$
If $D_R(f_1)(p_1)(f_1(x))=0$, then as before $p_1(x)=0$.  
We also have 
$$D_R(r_1)(p)(x)=\sum_{f_2(x_2)=f_1(x)}p(x,x_2)=0.$$
Similarly, one can show that $D_R(r_2)(p)=p_2$.
 
\end{proof}

\begin{prop}\label{pro:Glue} 
Let $R$ be a division which is also a zero-sum-free semiring.
		Consider the following diagram in $\catConvMon_R$:
\begin{equation}\label{DiagBull}
\begin{tikzcd}
	& M_1  \arrow[d,"f_1"] \\
M_2 \arrow[r,"f_2"] & N   
\end{tikzcd}
\end{equation}
For $(m_1,m_2) \in M_1 \times_N M_2$ the following are equivalent:
		\begin{enumerate}
			\item $(m_1,m_2)$ is weakly invertible.
			\item There are $P_1 \in D_R(M^\ast_1)$, $P_2 \in D_R(M^\ast_2)$ such that 
			$D_R(f_1)(P_1)=D_R(f_2)(P_2)$, 
			$\pi^{M_1}(P_1)=m_1$ and $\pi^{M_2}(P_2)=m_2$  
		\end{enumerate}
	\end{prop}
	\begin{proof}
		We have the following map between the {pull-back} squares:
\begin{equation}\label{M1NM2}
\begin{tikzcd}
M^\ast_1 \times_{N^\ast} M^\ast_2  \arrow[r]
 \arrow[d] \arrow[rrd] & M_1^\ast \arrow[d]  
\arrow[rrd,hookrightarrow] && \\
M^{\ast}_2 \arrow[r] \arrow[rrd,hookrightarrow] & N^{\ast}
 \arrow[rrd,hookrightarrow] & 
M_1 \times_N M_2  \arrow[r,"r_1"'] \arrow[d,"\;\;\;r_2"]& 
M_1 \ar[d,"f_1"] \\ 
&& M_2 \arrow[r, "f_2"'] & N 
\end{tikzcd}
\end{equation}
Applying $D_R$ to this diagram  induces 
\begin{equation}\label{PUllM1NM2}
\begin{tikzcd}
D_R(M^\ast_1 \times_{N^\ast} M^\ast_2)  \arrow[r]
 \arrow[d]
 & D_R( M_1 \times_N M_2) \arrow[d,"D_R(r_1) \times D_R(r_2)"]  \\
D_R(M^\ast_1) \times_{D_R(N^\ast)} D_R(M^\ast_2)   \arrow[r] &
D_R(M_1) \times_{D_R(N)} D_R(M_2)  
\end{tikzcd}
\end{equation}
The vertical maps in Diagram (\ref{PUllM1NM2}) are surjective by Lemma \ref{DRX1X2}.
Observe that $(M_1 \times_N M_2)^\ast =M^\ast_1 \times_{N^\ast} M^\ast_2$ by
the adjunction in (\ref{eq:adj-I}) and \cite[Theorem 6.3.1]{leinster2014basic}.
Composing  the Diagram (\ref{PUllM1NM2}) with
$\pi^{M_1} \times \pi^{M_2}:
D_R(M_1) \times_{D_R(N)} D_R(M_2) \to M_1 \times_N M_2$ we obtain
\begin{equation}\label{PUllM1NM3}
\begin{tikzcd}
D_R((M_1 \times_{N} M_2)^\ast) 	\arrow[r,hookrightarrow]	
\arrow[d,twoheadrightarrow] & D_R( M_1 \times_N M_2) 
\arrow[dd, bend left=100,"\pi^{M_1 \times_{N} M_2}"]
\arrow[d,"D_R(r_1) \times D_R(r_2)"']  \\
D_R(M^\ast_1) \times_{D_R(N^\ast)} D_R(M^\ast_2) 
\arrow[r] & 
D_R(M_1) \times_{D_R(N)} D_R(M_2) 		
\arrow[d,"\pi^{M_1} \times \pi^{M_2}"'] \\ 
&    M_1 \times_N M_2
\end{tikzcd}
\end{equation}
Since $\catMon$ is a complete category, using the proof of \cite[Theorem  5.6.5]{riehl2017category},
we obtain that
$$ 
\pi^{M_1 \times_{N} M_2}=(D_R(r_1) \times D_R(r_2) )\circ (\pi^{M_1} \times \pi^{M_2}).
$$
Therefore $(m_1,m_2) \in M_1 \times_N M_2$ is weakly invertible if 
and if $(m_1,m_2)$ is in the image of the map 
$$
 D_R(M^\ast_1) \times_{D_R(N^\ast)} D_R(M^\ast_2)  \to M_1 \times_N M_2.
$$
This happens if and only if there exists
$P_1 \in D_R(M^\ast_1)$ and $P_2 \in D_R(M^\ast_2)$ such that 
$D_R(f_1)(P_1)=D_R(f_2)(P_2)$ and 
$\pi^{M_1}(P_1)=m_1$, $\pi^{M_2}(P_2)=m_2$. 
\end{proof}
	%



\begin{rem}
{\rm

Let $X_1$ and $X_2$ be simplicial sets. Consider two simplicial set maps $f_i:Z\to X_i$ for $i=1,2$.
Let $Y$ be a simplicial group 
Applying Proposition \ref{pro:Glue} (together with Theorem \ref{Theo2} below) to the diagram of convex monoids
$$ 
\begin{tikzcd}
	& \catsSet(X_1,D_R(Y))  \arrow[d,"f_1^\ast"] \\
\catsSet(X_2,D_R(Y)) \arrow[r,"f_2^\ast"] & \catsSet(Z,D_R(Y))  
\end{tikzcd} 
$$
gives us a criterion on deciding whether a simplicial distribution $p\in \catsSet(X_1\sqcup_Z X_2, D_R(Y))$ is noncontextual. This is a generalization of \cite[Lemma 4.4]{okay2022simplicial}, in the case when $Y$ is a simplicial group, to the category of convex monoids. 
}
\end{rem}

\subsection{Strong invertibility} \label{sec:strong-invertibility}
	
{We introduce a stronger version of invertibility akin to strong contextuality (Definition \ref{suppp}) defined in terms of supports.}

	\begin{defn}\label{InvSup}
	{\rm
		Given an $R$-convex monoid $M$,
		the \emph{invertible support} of $m \in M$ is the following subset of
		$M^{\ast}$:
		$$
		\Isupp(m)=\{m' \in M^{\ast}:\, \exists  P \in D_R(M) \text{ such that }
		\pi(P)=m \text{ and } P(m')\neq 0 \}
		$$
We say that $m$ is 
		\emph{strongly non-invertible}
		if $\Isupp(m)=\emptyset$.
		}
	\end{defn}
	%
	%


{To study the properties of strong invertibility we will restrict to a special class of semirings.}

\begin{defn}\label{PropAB}
		{\rm
			A semiring $R$ is called {\it integral} 
			if $a\cdot b=0$  implies $a=0$ or $b=0$ if for all $a,b \in R$ (see {\cite[Definition 7]{jacobs2009duality}}). 
		}
	\end{defn}

{In this section the semiring $R$ will be zero-sum-free (Definition \ref{def:ZeroFreeSum}) and integral. 
}

	\begin{lem}\label{spp}
Let $R$ be a zero-sum-free, integral semiring, and $M$ be an $R$-convex monoid.
For $m_1,m_2 \in M$ we have 
$$\Isupp(m_1) \cdot \Isupp(m_2)  \subseteq  \Isupp(m_1\cdot m_2).$$   
Furthermore, if $m_1$ invertible, then 
$m_1 \cdot \Isupp(m_2)=\Isupp(m_1\cdot m_2)$.
	\end{lem}
	\begin{proof}
	Given $n_1 \in \Isupp(m_1)$, $n_2 \in \Isupp(m_2)$ there exists $P_1,P_2 \in D_R(M)$ such that
		$\pi(P_i)=m_i$ and $P_i(n_i)\neq 0$. Therefore
		 $\pi(P_1\ast P_2)=m_1\cdot m_2$ and 
		$$
		P_1\ast P_2(n_1\cdot n_2)=\sum_{x_1 \cdot x_2=n_1
			\cdot n_2}P_1(x_1)P_2(x_2)
		=P_1(n_1)P_2(n_2)+\cdots \neq 0.
		$$
This means that $n_1 \cdot n_2 \in \Isupp(m_1 \cdot m_2)$. 

Suppose now that $m_1 \in M^\ast$. Then
$m_1\cdot \Isupp(m_2) \subseteq \Isupp(m_1) \cdot \Isupp(m_2)
\subseteq \Isupp(m_1\cdot m_2)$. 
		Therefore $\Isupp(m_1 \cdot m_2)=m_1\cdot m_1^{-1}\cdot \Isupp(m_1\cdot m_2) \subseteq
	m_1 \cdot \Isupp(m_2)$.
	\end{proof}
	%
	%
	%

	%

{As an immediate consequence of this observation we have the following.}	
 	
	\begin{corollary}\label{msNInmsNI}
	Let $R$ be a zero-sum-free, integral semiring, and $M$ be an $R$-convex monoid.
Let $n\in M^{\ast}$ and $m \in M$. 
Then the following are equivalent:
\begin{enumerate}
\item $m$ is strongly non-invertible.
\item $n\cdot m$ is strongly non-invertible.
\item $m\cdot n$
is strongly non-invertible.
\end{enumerate}
	\end{corollary}

	\begin{lem}\label{sup(fm)}
Let $R$ be a zero-sum-free semiring.
For a morphism $f:M_1 \to M_2$  in $\catConvMon_R$ and $m\in M_1$ we have
$f(\Isupp(m)) \subseteq \Isupp f(m)$.
\end{lem}
\begin{proof}
	Given $n \in \Isupp(m)$ there exists $P \in D_R(M_1)$ such that
	$\pi^{M_1}(P)=m$ and $P(n)\neq 0$.  
	We have
	$$
	\pi^{M_2}(D_R(f)(P))=f(\pi^{M_1}(P))=f(m) 
	$$
	and
	$$
	D_R(f)(P)(f(n))= \sum_{f(n')=f(n)} P(n')=P(n)+\cdots \neq 0.
	$$
	Therefore $f(n) \in \Isupp(f(m))$. 
\end{proof}
%
Next result follows immediately from Lemma \ref{sup(fm)}.
\begin{corollary}\label{ForAPP1}
Let $R$ be a zero-sum-free semiring.
For a morphism $f:M_1 \to M_2$  in $\catConvMon_R$ and $m\in M_1$ we have the following:
\begin{enumerate}
\item If
	$\Isupp (f(m)) \cap f(M^{\ast}_1)=\emptyset$, then $m$ is strongly
	non-invertible.
\item If $f(m)$ is strongly non-invertible, then $m$ is also strongly non-invertible.
\end{enumerate}

\end{corollary}
%

%

	%

%
\subsection{Invertible fraction} \label{sec:invertible-fraction}
We now give the definition of the invertible fraction for  real convex monoids.
\begin{defn}\label{InvFr}
{\rm
Let $M$ be a real convex monoid. 
The \emph{invertible
		fraction} of $m \in M$, denoted by $\IF(m)$, is the supremum of	
	$$
		 \{ \sum_{m' \in M^{\ast}} P(m'):\, P \in D_R(M)
	\text{ such that } \pi(P)=m\}.
	$$
	The non-invertible fraction of $m$
	is defined to be 
	$\NIF(m)=1-\IF(m)$.
	}
\end{defn}

\begin{prop}\label{IFm0}
	Let $M$ be a real convex monoid. An element $m\in M$ is \emph{strongly non-invertible} if and only if $\IF(m)=0$.
\end{prop}
\begin{proof}
	The element $m$ is strongly non-invertible if and only if $\Isupp(m)=\emptyset$. The set $\Isupp(m)$ is empty if and only if
every $P \in D(M)$ such that $\pi(P)=m$ satisfies 
	$P(m') = 0$ for all $m' \in M^{\ast}$.
	This is equivalent to every $P \in D(M)$ such that $\pi(P)=m$ satisfying 
	$\sum_{m'\in M^{\ast}}P(m') = 0$. 
	This   means that $\IF(m)=0$. 
\end{proof}

\begin{prop}\label{deg(fg)}
	For $m_1,m_2 \in M$ we have 
	$$
	\IF(m_1 \cdot m_2) \geq \IF(m_1)\cdot \IF(m_2).
	$$
\end{prop}
\begin{proof}
	Let $P_1,P_2 \in D_R(M)$  such that $\pi(P_i)=m_i$ for $i=1,2$. Then 
	$\pi(P_1 \ast P_2)=m_1 \cdot m_2$ and
	$$
	\begin{aligned}
	\sum_{m' \in M^{\ast}} (P_1\ast P_2)(m') &=
	\sum_{m' \in M^{\ast}} \sum_{m'_1 \cdot
		m'_2=m'} P_1(m'_1)P(m'_2) \\
	&=\sum_{m'_1\cdot m'_2\in M^{\ast}} P_1(m'_1)P_2(m'_2)\\
	&\geq   \sum_{m'_1 , m'_2\in M^{\ast}} P_1(m'_1)P_2(m'_2)\\
	&=
	\sum_{m'_1\in M^{\ast}} P(m'_1)
	\sum_{m'_2\in M^{\ast}} P(m'_2).
	\end{aligned} 
	$$
	Therefore $\IF(m_1 \cdot m_2) \geq \IF(m_1)\cdot \IF(m_2)$.
\end{proof}
\begin{corollary}
	Given $n\in M^\ast$ and $m\in M$ we have $\IF(n\cdot
	m)=\IF(m)$.
\end{corollary}

\begin{proof}
	By Proposition \ref{deg(fg)} we have $\IF(n\cdot m)\geq \IF(n)\IF(m)=\IF(m)$.
	On the other hand, $m=n^{-1}\cdot(n\cdot m)$. Therefore 	$\IF(m) \geq \IF(n \cdot m)$.
\end{proof}

\begin{prop}\label{IFf}
	Given a morphism $f:M_1 \to M_2$  in $\catConvMon$ and  $m \in
	M$ we have
	$\IF(f(m))\geq \IF(m)$.
\end{prop}
\begin{proof}
	For $P \in D(M_1)$ such that {$\pi^{M_1}(P)=m$} we have 
	$\pi^{M_2}(D(f)(P))=f(\pi^{M_1}(P))=f(m)$ and
	$$
	\sum_{n' \in M^{\ast}_2} D(f)(P)(n')=\sum_{n' \in M^{\ast}_2}\sum_{f(m)=n'}P(m)
	\geq \sum_{m'\in M_1^{\ast}} P(m').
	$$
\end{proof}
%
%
\begin{prop}\label{Lin}
	For $m \in M$ and $\pi(P)=m$ we have
	$$
	\IF(m)\geq \sum_{x \in M}P(x)\IF(x).
	$$
\end{prop}
\begin{proof}
Let us write $m_1,m_2,\cdots,m_k$ for the distinct elements in 
	$\{x \in M ~|~P(x) \neq 0\}$.
	Given $\epsilon >0$, for every $1\leq i \leq k$ we choose
	$P_i \in D(M)$	such that $\pi(P_i)=m_i$ and
	$\IF(m_i)-\epsilon<\sum_{m' \in M^{\ast}}P_i(m')$.
	We have $\sum_{i=1}^{k}P_i(m_i)=1$. Thus we can define
	$Q=\sum_{i=1}^{k}P(m_i)P_i \in D(M)$.
	In fact, $Q=\mu_M(\tilde{P})$, where $\tilde{P} \in D(D(M))$ is
	defined as the follows:
	$$
	\tilde{P}(S)=
	\begin{cases}
		P(m_i)  & S=P_i \\
		0 & ~\text{otherwise.}
	\end{cases}
	$$
	On the other hand, $D(\pi)(\tilde{P})\in D(M)$ and for $x\in 	M$ we have
	$$
	D(\pi)(\tilde{P})(x)=\sum_{\pi(q)=x }\tilde{P}(q)=
	\sum_{i:\,\pi(P_i)=x}\tilde{P}(P_i)=
	\begin{cases}
		P(m_i)  & x=m_i \\
		0 & \text{otherwise.}
	\end{cases}
	$$
	We obtain that $D(\pi)(\tilde{P})=P$. Thus by the right-hand Diagram in (\ref{AAA}), we see that $\pi(Q)=\pi(P)$. Therefore $\pi(Q)=m$. Using this and by 
	our choice of $P_1, \cdots, P_k$, we obtain
	$$
	\begin{aligned}
		\sum_{m' \in M^{\ast}}Q(m') &= \sum_{m' \in
			M^{\ast}}
		\sum_{i=1}^{k}P(m_i)P_i(m')\\
		 & = 
		\sum_{i=1}^{k}\sum_{m' \in M^{\ast}}
		P(m_i)P_i(m')\\
		&=\sum_{i=1}^{k}P(m_i)(\sum_{m'
			\in M^{\ast}}
		P_i(m'))\\
		& > \sum_{i=1}^{k}P(m_i)(\IF(m_i)-\epsilon)\\
		&=
	\sum_{i=1}^{k}P(m_i)\IF(m_i)-\epsilon.
	\end{aligned}
	$$
	%
	Hence we proved that for every $\epsilon >0$ there exists $Q \in
	D_R(M)$
	such that $\pi(Q)=m$ and
	$$
	\sum_{m' \in M^{\ast}}Q(m')> 
	\sum_{m\in M}P(m)\IF(m)-\epsilon.
	$$
	This gives  the desired result.
\end{proof}
\begin{rem}\label{rem:IF>=PIF}
{\rm
{Theorem \ref{thm:ContextualSheaf} implies that the notion of contextuality for simplicial distributions and the corresponding notion for presheaves of distributions provided in \cite{abramsky2011sheaf} coincide. Later in Corollary \ref{NCFFF=IFFF} we will show that noncontextual fraction of a simplicial distribution $p:X\to D(Y)$ is equal to its invertible fraction (here we assume $Y$ is a simplicial group).
With these observations Proposition \ref{Lin} generalizes the first inequality of \cite[Theorem 2]{Abramsky_2017} satisfies by the contextual fraction of a presheaf of distributions.
}
%
}
\end{rem}



	%
%
	 
%

{Next we provide an alternative characterization of invertible fraction, which will be useful when comparing this notion to noncontextual fraction.}

\begin{prop}\label{IFFF2}
Let $M$ be a real convex monoid.
For $m\in M$, the invertible fraction 
	$\IF(m)$ is equal to the supremum of
	\begin{equation}\label{SSS}
		\{ \alpha \in [0,1]:\, m_1 \text{ is weakly invertible},\;m_2 \in M \text{ such that } m=\alpha m_1+ (1-\alpha)m_2\, \}.
	\end{equation}
\end{prop}
\begin{proof}
We begin by an observation.
For $P \in D(M)$, we define  
	$$
	I(P)(m)=
	\begin{cases}
		\frac{P(m)}{\sum_{n \in M^\ast}P(n)}  & \text{if} ~ m \in M^\ast \\
		0 & \text{otherwise,}
	\end{cases}
	$$
	and 
	$$
	NI(P)(m)=
	\begin{cases}
		0 & \text{if} ~ m \in M^\ast \\
		\frac{P(m)}{\sum_{n \in M-M^\ast}P(n)}  & \text{otherwise.}
	\end{cases}
	$$	
Then 
we have 
\begin{equation}\label{eq:pinpin}
P = \sum_{n \in M^\ast}P(n) I(P) + \sum_{m \in M-M^\ast}P(m) NI(P).
\end{equation}	

Now, let us denote the set in Equation (\ref{SSS}) by $A$. 
For $\alpha \in A$, 
	there exists $P_1 \in D(M^\ast)$ and $m_2 \in M$ such that
	$$
	m=\alpha \pi^M(P_1) + (1- \alpha) m_2=
	\alpha \pi^M(P_1) + (1- \alpha) \pi^M(\delta^{m_2}).
	$$
{By Proposition \ref{Conv=ConvR} the structure map $\pi^M$ is a morphism in $\catConv$.}
Therefore we obtain   
	$$
	m=\pi^M(\alpha P_1 + (1- \alpha) \delta^{m_2}).
	$$
	Observe that 
	$$
	\sum_{m' \in M^\ast}(\alpha P_1 + (1- \alpha) \delta^{m_2})(m')
	=\alpha\sum_{m' \in M^\ast}P_1(m') +
	(1-\alpha)\sum_{m' \in M^\ast}{\delta^{m_2}}(m') \geq
	\alpha\sum_{m' \in M^\ast}P_1(m').  
	$$
Since $P \in D(M^\ast)$ we have  $\sum_{m' \in M^\ast}P_1(m')=1$, and therefore
	$$
	\sum_{m' \in M^\ast}(\alpha P_1 + (1- \alpha) \delta^{m_2})(m')\geq
	\alpha.
	$$ 
	This yields that $\IF(m) \geq \sup A$.
	Now given $P \in D(M)$ such that $\pi(P)=m$,  Equation (\ref{eq:pinpin}) implies that
	$$
	P = \sum_{n \in M^\ast}P(n) I(P) +
	\sum_{m \in M-M^\ast}P(m) NI(P) .
	$$
	Applying $\pi^M$ we obtain
	$$
	m=\sum_{n \in M^\ast}P(n) \pi^M(I(P)) +
	\sum_{m \in M-M^\ast}P(m) \pi^M(NI(P)).
	$$
	Note that $\pi^M (I(P))$ is weakly invertible, thus 
	$\sup A \geq \IF(m)$.
\end{proof}
%

%

%
%



\section{Contextuality and invertibility}\label{Sec: SimDis} 

{Recall that a simplicial distribution $p:X\to D_R(Y)$ is called noncontextual if it lies in the image of $\Theta_{X,Y}:D_R(\catsSet(X,Y))\to \catsSet(X,D_R(Y))$ (Definition \ref{def:contextual-morphism}).
When $Y$ is a simplicial group the set of simplicial distributions $\catsSet(X,D_R(Y))$ is a convex monoid (Corollary \ref{SCM}). In this monoid we can consider those elements that are weakly invertible (Definition \ref{weakk}). In this section we prove our main result, which is the equivalence of the two notions for simplicial distributions.
We also prove a similar equivalence between their strong versions.  
}

\subsection{Contextuality and weak invertibility}\label{sec:cont-weak-inv}
{
In this section we establish the equivalence of the notions of noncontextuality and weak invertibility. 
}

\begin{lem}\label{TetaMorph}
Let $X$ be 
a simplicial set 
and $Y$ be a simplicial monoid.
Then the map
	$\Theta_{X,Y}:D_R(\St(X,Y)) \to \St(X,D_R(Y))$ is a homomorphism of monoids.
\end{lem}
\begin{proof}
	Given $p,q \in D_R(\St(X,Y))$,
	and  $x \in X_n$, $y \in Y_n$, we have {
	$$
	\begin{aligned}
	\Theta(p \ast q)_n(x)(y) &=
	\sum_{\varphi\in \catsSet(X,Y):\,\varphi_n(x)=y}(p \ast q)(\varphi) \\
	&= \sum_{ \varphi_n(x)=y}\sum_{\psi \cdot
		\psi' = \varphi }p(\psi) q(\psi')\\
	&= 	\sum_{\psi_n(x)\cdot \psi'_n(x)=y}p(\psi) q(\psi') \\
	&=\sum_{y_1 \cdot y_2=y}
		~\sum_{\psi_n(x)=y_1,\,\psi'_n(x)=y_2}p(\psi)q(\psi') \\
		&= \sum_{y_1 \cdot y_2=y}
		\left(\sum_{\psi_n(x)=y_1}p(\psi)\right)
		\left(\sum_{\psi'_n(x)=y_2}q(\psi')\right)\\
		&=
\sum_{y_1 \cdot y_2=y}\Theta(p)_n(x)(y_1) \cdot \Theta(q)_n(x)(y_2)\\
&=(\Theta(p)_n(x) \ast \Theta(q)_n(x))(y) \\
&= (\Theta(p) \cdot \Theta(q))_n(x)(y).
	\end{aligned}
	$$}
	%
	%
	%
		%
		%
\end{proof}
\begin{cor}\label{SCM}
The functor $\St(-,-) : \St^{op} \times \St \to \catSet$ restricts to a functor
$$
\St(-,-): \St^{op} \times s\catConvMon_R \to \catConvMon_R
$$
In particular, for a simplicial set $X$ and a simplicial monoid $Y$ the set {$\catsSet(X,D_R(Y))$} of morphisms is an $R$-convex monoid. 	
\end{cor}
\begin{proof}
The structure map $\pi^{\St(X,Y)}:D_R(\catsSet(X,Y))\to \catsSet(X,Y)$ is a homomorphism of monoids since it is the composite of $\Theta_{X,Y}:D_R(\catsSet(X,Y)) \to \catsSet(X,D_RY)$, which is a homomorphism by Lemma \ref{TetaMorph}, with the homomorphism $(\pi^Y)_\ast: \catsSet(X,D_RY) \to \catsSet(X,Y)$.
In addition, by Proposition \ref{Hommm}    
$\St(X,Y)$ is an $R$-convex set. 
It remains to show that the restriction is compatible with morphisms. This follows from Propsition \ref{Hommm} and the fact that the functor $\St(-,-) : \St^{op} \times \St \to \catSet$ restricts to a functor
$\St(-,-): \St^{op} \times s\catMon \to \catMon$.
\end{proof}

%

{
Next we describe the monoid structure of $\catsSet(X,D(N\ZZ_2))$ when $X$ is one of the spaces in Example \ref{ex:triangle} and \ref{Ex:Sqqq}. 
}
	
\begin{example}\label{Ex3 : CM}
{\rm
First, let us describe the product for the triangle $X=\Delta[2]$.
Given simplicial distributions $p,q$  represented by the boxes as in Example \ref{ex:edge-dist}
$$  
\begin{tabular}{|c|c|}  
\hline
  & $y$  \\ 
\hline
$x$ & {\begin{tabular}{cc} $p_1$ & $p_2$ \\ $p_3$ & $p_4$ \\ \end{tabular}}  \\
\hline
\end{tabular} 
\;\;\;\;\;\;
\begin{tabular}{|c|c|}  
\hline
  & $y$  \\ 
\hline
$x$ & {\begin{tabular}{cc}$q_1$ & $q_2$ \\ $q_3$ & $q_4$ \\
 \end{tabular}}  \\
\hline
\end{tabular} 
$$
the product $p\cdot q$ is represented by the box
\begin{equation}\label{eq:triangle-product}
\begin{tabular}{|c|c|}  
\hline
  & $y$  \\ 
\hline
$x$ & {\begin{tabular}{cc}
$p_1 \cdot q_1 + p_2 \cdot q_2 + p_3 \cdot q_3 + p_4 \cdot q_4$
&~~ $p_1 \cdot q_2 + p_2 \cdot q_1 + p_3 \cdot q_4 + p_4 \cdot q_3$
\\ 
$p_1 \cdot q_3 + p_3 \cdot q_1 + p_2 \cdot q_4 + p_4 \cdot q_2$
&~~ $p_1 \cdot q_4 + p_2 \cdot q_3 + p_3 \cdot q_2 + p_4 \cdot q_1$ \\
 \end{tabular}} \\
\hline
\end{tabular} 
\end{equation}
Similarly, we can describe the product for the square $X=\Delta[0]\ast X_\Sigma$.
The identity element in 
$\St(\Delta[0]\ast X_\Sigma ,DN\zz_2)$ is given by
$$
\begin{tabular}{|c|c|c|c|c|} 
\hline
  & $y_0$ & $y_1$ \\ 
\hline
$x_0$ & {\begin{tabular}{cc} $1$ & $0$ \\ $0$ & $0$ \\ \end{tabular}} & {\begin{tabular}{cc} $1$ & $0$ \\ $0$ & $0$ \\ \end{tabular}} \\ 
\hline
$x_1$ & {\begin{tabular}{cc} $1$ & $0$ \\ $0$ & $0$ \\ \end{tabular}} & {\begin{tabular}{cc} $1$ & $0$ \\ $0$ & $0$ \\ \end{tabular}} \\  
\hline
\end{tabular} 
$$
The product can be computed considering one box at a time, labeled by $(x_i,y_j)$, and using the formula in Equation (\ref{eq:triangle-product}).
}
\end{example}

{Our main result in this section connects noncontextuality and weak invertibility. 
}

\begin{thm}\label{Theo1}
	Given a simplicial set $X$ and a simplicial group $Y$.
	If the distribution  $p \in \St(X,D_R(Y))$ is noncontextual then
	$p$ is weakly invertible (Definition \ref{weakk}).  
\end{thm}
\begin{proof}
	We have the following commutative diagram:
\begin{equation}\label{Isommm}
\begin{tikzcd}
(\St(X,D_R(Y)))^\ast   \arrow[r,hook] & \St(X,D_R(Y))  \\
		\St(X,Y)  
 \arrow[u,hook,"(\delta_{Y})_\ast"]  
\arrow[ru,hook,"(\delta_{Y})_\ast"']  
\end{tikzcd}
\end{equation}
By Proposition \ref{ThetaUniqqq} 
$\Theta_{X,Y}$ is the transpose of 
$(\delta_Y)_\ast$ with respect to the adjunction $D_R:\catMon \adjoint \catConvMon_R:U$.
Similarly, the map 
$\tilde{\pi}^{\St(X,D_R(Y))}$ is the transpose of the inclusion homomorphism 
$(\St(X,D_R(Y)))^\ast \hookrightarrow \St(X,D_R(Y))$ with respect to the same adjunction. 	
Therefore we obtain the following commutative diagram:
\begin{equation}\label{Importtt}
\begin{tikzcd}
D_R((\St(X,D_R(Y)))^\ast)   
\arrow[rrr,"{\tilde\pi}^{\St(X,D_R(Y))}"] &&& \St(X,D_R(Y))  \\
D_R(\St(X,Y)) 
 \arrow[u,hook,"D_R((\delta_{Y})_\ast)"]  
\arrow[rrru,"\Theta_{X,Y}"']  
\end{tikzcd}
\end{equation}
This diagram gives the desired result.
\end{proof}
%
%
The following example shows that the opposite direction of Theorem \ref{Theo1} does not hold in general. However, in Theorem \ref{Theo2} we will show that it holds for a large class of semirings.
%
\begin{example}\label{ex:D1-DeltaZ2}
{\rm
Let $X$ be the simplicial circle $S^1=\Delta[1]/\partial \Delta[1]$. It has one vertex $x$, and one nondegenerate $1$-simplex $\sigma$.
A distribution $p \in \St(X,D_{\rr}(\Delta_{\zz_2}))$ is 
given by a tuple {$(p^{00},p^{01},p^{10},p^{11})$} 
where $p^{00}+p^{10}=d_0(p_\sigma)=d_1(p_\sigma)=p^{00}+p^{01}$.
$$
\begin{tabular}{|c|c|}  
\hline
  & $x$  \\ 
\hline
$x$ & {\begin{tabular}{cc} $p^{00}$ & $p$ \\ $p$ & $p^{11}$ \\ \end{tabular}}  \\
\hline
\end{tabular} 
$$
Deterministic distributions on $(S^1,\Delta_{\ZZ_2})$ are given by $(1,0,0,0)$ and $(0,0,0,1)$. Therefore the
distribution  $(1,2,2,-4)$ is contextual. But this distribution is weakly invertible, in fact even invertible, with inverse given by
$(\frac{11}{35},\frac{2}{7},\frac{2}{7},\frac{4}{35})$. 
}
\end{example}

\begin{lem}\label{GtoDRG}
Let $R$ be a 
zero-sum-free, integral
	semiring 
	and $G$ be a group.
Then $(D_R(G))^{\ast}$ is equal to the image of $\delta_G : G \to D_R(G)$. 
\end{lem}
\begin{proof}
{The image of $\delta_G$ is contained in $(D_R(G))^*$, since this map is a homomorphism of monoids.
For the converse consider 
$p=\sum_{g \in G}\alpha_g\delta^g$ 
with inverse $q=\sum_{g \in G}\beta_g\delta^g$, that is
$$
p\ast q(h) = \sum_{g_1g_2=h} p(g_1) q(g_2) = 
\left\lbrace
\begin{array}{cc}
1 & h=e_G\\
0 & h\neq e_G.
\end{array}
\right.
$$
The case $h=e_G$ implies that there exists $g\in G$ for which $q(g)\neq 0$. On the other hand, the second case gives us that $p(hg^{-1})=0$ for all $h\neq e_G$, since $R$ is zero-sum-free and integral. Therefore $p$ is the  delta distribution $\delta^{g^{-1}}$.
}
\end{proof}

\begin{lem}\label{HomInvvv}
Let $R$ be a zero-sum-free and integral
	semiring. 
Given a simplicial set $X$ and a simplicial group $Y$ the set $(\St(X,D_R(Y)))^\ast$ of {units} is the image of the following map:
	$$
	(\delta_Y)_{\ast} : \St(X,Y) \to \St(X,D_R(Y)).
	$$
\end{lem}
\begin{proof}
{Since $(\delta_Y)^*$ is a homomorphism of monoids we have} that $\Image((\delta_Y)_{\ast}) \subseteq (\St(X,D_R(Y))^\ast$. 
For the other direction, consider
$p,q \in \St(X,D_R(Y))$ such that 
	$p\cdot q =e_{\St(X,D_R(Y))}$. 
This means that 
we have $p(x) \ast q(x)=e_{D_R(Y_n)}$ for 
$x \in X_n$.
By Lemma \ref{GtoDRG}  
there exists $y(x) \in Y_n$ such that 
$p(x) =\delta^{y(x)}$. 
Then we define {$\varphi_n:X_n\to Y_n$ by} $\varphi_n(x)=y(x)$.
Then compatibility of $p$ with the simplicial structure maps and 
	$$
	d_i(\delta^y)=\delta^{d_i(y)},\;\;
	s_j(\delta^y)=\delta^{s_j(y)}
	$$
{implies that $\varphi:X\to Y$ defined in degree $n$ by $\varphi_n$ is a simplicial set map such that $(\delta_Y)_\ast(\varphi)=p$. }
\end{proof}

\begin{thm}\label{Theo2}
Let $R$ be a 
zero-sum-free and integral
	semiring. 
Given a simplicial set $X$ and a simplicial group $Y$, a distribution  $p \in \St(X,D_R(Y))$ is noncontextual if and only if $p$	is weakly invertible.  
\end{thm}
\begin{proof}
This follows from Diagram (\ref{Importtt}) and Lemma \ref{HomInvvv}.
\end{proof}

\begin{corollary}\label{NCFFF=IFFF} 
Let $X$ be a simplicial set and $Y$ be a simplicial group.  
For $p  \in \St(X,D(Y))$ we have 
	$$
	\NCF(p) = \IF(p).
	$$ 
\end{corollary}
\begin{proof}
Follows from Definition \ref{CFFF}, Proposition \ref{IFFF2}, and Theorem \ref{Theo2}.  
\end{proof}

\subsection{Strong contextuality and strong invertibility} \label{sec:strong-cont-inv}
 
{In this section we make the connection between strong contextuality and strong non-invertibility when $R=\RR_{\geq 0}$.

Let $X$ be a simplicial set. A simplex $x\in X_n$ is called degenerate if $x$ belongs to {$\cup_{j=0}^{n-1}s_j(X_{n-1})$}; otherwise it is called nondegenerate.
}

\begin{lem}\label{Isupp=supppp}
Let $X$ be a simplicial set with finitely many
nondegenerate simplices and $Y$ be a simplicial group. For  $p \in \St(X,D(Y))$ we have
$$
\Isupp(p)=(\delta_Y)_\ast(\supp(p)).
$$
\end{lem}

\begin{proof} 
A distribution $q \in \Isupp(p)$  is invertible, and by 
Lemma \ref{HomInvvv} there exists $\varphi \in \St(X,Y)$ such that 
$q=\delta_Y \circ \varphi$. 
Also there exists 
$Q \in D(\St(X,D(Y))$ such that $\pi^{\St(X,D(Y)}(Q)=p$ and 
$Q(\delta_Y \circ \varphi)=Q(q) > 0$. 
Using Equation (\ref{piForm}), for $x \in X_n$ we obtain 
$$
\begin{aligned}
p_n(x)(\varphi_n(x)) &=
\sum_{p' \in \catsSet(X,D_R(Y)) }Q(p')\,p_n'(x)(\varphi_n(x)) \\
&\geq Q(\delta_Y \circ \varphi)(\delta_{Y_n}\circ \varphi_n(x))(\varphi_n(x))\\
&=Q(\delta_Y \circ \varphi)\delta^{\varphi_n(x)}(\varphi_n(x)) \\
&=Q(\delta_Y \circ \varphi) > 0.
\end{aligned}
$$
Therefore $\varphi \in \supp(p)$ and 
$q \in (\delta_Y)_\ast(\supp(p))$.

For the converse inclusion, we will show that $\delta_Y \circ \psi \in \Isupp(p)$ for $\psi \in \supp(p)$.
We define 
$$
\alpha =\min\{p_n(x)(\psi_n(x)):\, x \in X_n ,\, n \geq 0\}.
$$
Observe that $\alpha >0$ since $X$ has finitely many simplices 
and for   $x \in X_n$ we have both 
$$p_{n+1}(s_i(x))(\psi_{n+1}(s_i(x)))\;\; \text{ and }\;\; p_{n-1}(d_j(x))(\psi_{n-1}(d_j(x)))$$ 
are greater than or equal to $p_n(x)$.
%
%
If $\alpha=1$, which implies that $p=\delta_Y \circ \psi$, then 
$\delta_Y \circ \psi \in \Isupp(p)$. 
Then let us suppose that $\alpha <1$. For   $n\geq 0$, we define $q_n: X_n \to D_R(Y_n)$ by
$$
q_n(x)=\frac{p_n(x)-\alpha\delta^{\psi_n(x)}}{1-\alpha},\;\;\;x \in X_n. 
$$
{We need to verify that indeed $q_n(x)\in D_R(Y_n)$:}
Observe that for $y \in Y_n$, we have
$$
q_n(x)(y) =
\begin{cases}
\frac{p_n(x)(\psi_n(x)) -\alpha}{1-\alpha}  & \text{if} ~ y=\psi_n(x),\\
\frac{{p}_n(x)(y)}{1-\alpha} & \text{otherwise.}
\end{cases}
$$
By definition of $\alpha$, we have  $q_n(x)(y) \geq 0$
for all $x \in X_n$, $y \in Y_n$. 
In addition,  we have
$$
\begin{aligned}
\sum_{y\in Y_n}q_n(x)(y) &=\sum_{y\in Y_n}\frac{p_n(x)(y)-\alpha\delta^{\psi_n(x)}(y)}{1-\alpha} \\
&=\frac{1}{1-\alpha}\left(\sum_{y\in Y_n}p_n(x)(y)
-\alpha \sum_{y\in Y_n}\delta^{\psi_n(x)}(y) \right) \\
&=\frac{1}{1-\alpha}(1- \alpha )=1.
\end{aligned}
$$
Next we  prove that the collection of maps $\{q_n\}_{n \geq 0}$ form a simplicial set map $q : X \to D_R(Y)$.
Given $x\in X_n$ and $y \in Y_{n-1}$, we will show that 
$(D_R(d_j) \circ q_n(x))(y)= (q_{n-1} \circ d_j(x))(y)$.
We begin with the case that $ y \neq \psi_{n-1}(d_j(x))$: We have
$$
\begin{aligned}
q_{n-1}(d_j(x))(y)&=\frac{p_{n-1}(d_j(x))(y)}{1-\alpha}\\
&=\frac{1}{1-\alpha}D(d_j)(p_n(x))(y)\\
&=
\sum_{d_j(y')=y}\frac{p_n(x)(y')}{1- \alpha} \\
&= \sum_{d_j(y')=y}q_n(x)(y') \\
&= D_R(d_j) (q_n(x))(y).
\end{aligned}
$$
In the fourth line we used the observation that {$y'$ with $d_j(y')=y$ satisfies $y'\neq \psi_n(x)$},
otherwise we would have 
$y=d_j(y')=d_j(\psi_n(x))=\psi_{n-1}(d_j(x))$.
Next assume that 
$ y = \psi_{n-1}(d_j(x))=d_j(\psi_n(x))$. Then we have 
$$
\begin{aligned}
q_{n-1}(d_j(x))(y) &=
\frac{p_{n-1}(d_j(x))(y)-\alpha}{1-\alpha} \\
&=\frac{D_R(d_j)(p_{n}(x))(y)-\alpha}{1-\alpha} \\
&=\frac{\sum_{d_j(y')=y}p_n(x)(y')-\alpha}{1- \alpha}\\
&=\frac{p_n(x)(\psi_n(x))-\alpha}{1- \alpha}+
\sum_{d_j(y')=y,\,y' \neq \psi_n(x)}\frac{p_n(x)(y')}{1- \alpha}\\
&=q_n(x)(\psi_n(x)) + \sum_{d_j(y')=y,\,y' \neq \psi_n(x)}q_n(x)(y') \\
&=\sum_{d_j(y')=y}q_n(x)(y')\\
&=D_R(d_j) (q_n(x))(y).
\end{aligned}
$$
Now we   define $Q \in D(\St(X,D(Y)))$ by 
$$
Q(p')=
\begin{cases}
1-\alpha & \text{if}~p'=q \\
\alpha & \text{if} ~ p'=\delta_Y \circ \psi \\
0  & \text{otherwise.}
\end{cases}
$$
For $x \in X_n$, we have 
$$
\begin{aligned}
\pi^{\St(X,D(Y))}(Q)_n(x) &=(1-\alpha)q_n(x)+\alpha (\delta_{Y_n} \circ \psi_n)(x) \\ 
&=(1-\alpha)\frac{p_n(x)-\alpha\delta^{\psi_n(x)}}{1-\alpha}
+\alpha \delta^{\psi_n(x)} \\
&=p_n(x).
\end{aligned}
$$
We showed that $\pi^{\St(X,D(Y))}(Q)=p$, while 
$Q(\delta_{Y} \circ \psi)=\alpha \neq 0$, which implies that 
$\delta_{Y} \circ \psi \in \Isupp(p)$.
\end{proof}

\begin{corollary}\label{SC=SNI}
Let $X$ be a simplicial set with finitely many
nondegenerate simplices, $Y$ be a simplicial group and $p  \in \St(X,D(Y))$. 
\begin{enumerate}
\item $p$ is 
strongly contextual if and only if $p$ is strongly non-invertible.
\item $p$ is strongly contextual if 
and only if $\CF(p)=1$. 
\end{enumerate} 
\end{corollary}
%
%
\begin{proof}
The first part follows directly from Lemma \ref{Isupp=supppp}. The second part follows from the first part together with Proposition \ref{IFm0}, and Corollary
\ref{NCFFF=IFFF}.
\end{proof}
\begin{rem}\label{rem:strongly-contextual-CF1}
{\rm
Theorem \ref{thm:ContextualSheaf} together with Remark \ref{rem:strong-fraction} implies that Corollary \ref{SC=SNI} part (2) generalizes Proposition 6.3 in \cite{abramsky2011sheaf}.
}  
\end{rem}

\subsection{{Extremal simplicial distributions}}
 \label{sec:extremal}

\begin{defn}\label{def:vertex}
{\rm
Let $(X,\pi^{X})$  be an $R$-convex set.
An element $x\in X$ is 
called a \emph{vertex}, or an \emph{extreme point}, if $x$ has a unique preimage under $\pi^X$. 
}
\end{defn}
%
%
Recall the definition of deterministic simplicial distribution from Definition \ref{def:deterministic}.

\begin{prop}\label{Detvert}
Let $R$ be a zero-sum-free, integral semiring. 
For simplicial sets $X,Y$ every deterministic distribution in $\St(X,D_R(Y))$ is a vertex.
\end{prop}
\begin{proof}
Let $\varphi \in \St(X,Y)$. Suppose that we have 
$Q \in D_R(\St(X,D_R(Y)))$ such that 
$\pi^{\St(X,D_R(Y))}(Q)=\delta_Y \circ \varphi$. This means that 
$$
\sum_{p\in \St(X,D_R(Y))}Q(p)p=\delta_Y \circ \varphi.
$$
Given $q \in \St(X,D_R(Y))$, not equal to $\delta_{Y}\circ \varphi$,  there exists $x \in X_n$ and $ y \in Y_n$ such that 
$y \neq \varphi_n(x)$ and $q_n(x)(y) \neq 0$. We have 
$$
\sum_{p\in \St(X,D_R(Y))}Q(p)p_n(x)(y)=\delta_{Y_n} \circ \varphi_n(x)(y)= \delta^{\varphi_n(x)}(y)=0.
$$
Since $R$ is a zero-sum-free semiring, we obtain that 
$Q(q)q_n(x)(y)=0$. Since $R$ is integral, we have 
$Q(q)=0$.
This implies  {$Q=\delta^{\delta_{Y}\circ \varphi}$}.
\end{proof}

\begin{prop}\label{f-1vert}
Let $R$ be a zero-sum-free semiring and 
$f : X \to Y$ be a morphism in $\catConv_R$.
If $y\in Y$ is a vertex,
then every vertex of the $R$-subconvex set $f^{-1}(y)$ is a vertex of $X$.
\end{prop}
\begin{proof}
Let $x$ be a vertex of $f^{-1}(y)$ and $p \in D_R(X)$
such that $\pi^X(p)=x$. 
Then 
$$
\pi^Y(D_R(f)(p))=f(\pi^X(p))=f(x)=y.
$$
Since $y$ is a vertex of $Y$, we obtain that $D_R(f)(p)=\delta^y$. 
This means that if $y' \neq y $ then 
$D_R(f)(p)(y')=\sum_{f(x')=y'}p(x')=0$. 
Since $R$ is zero-sum-free we obtain that $p(x')=0$ for all 
$x' \notin f^{-1}(y)$. In other words, $p \in D_R(f^{-1}(y))$.
Finally since
 $y$ is a vertex of $f^{-1}(y)$, we obtain that $p=\delta^y$.
\end{proof}

{Main application of this observation to simplicial distributions is as follows.

\begin{cor}
Let $f:Z\to X$ be a map of simplicial sets. Consider the morphism $f^*:\catsSet(X,D_R(Y))\to \catsSet(Z,D_R(Y))$
in $\catConv_R$. If $p\in \catsSet(Z,D_R(Y))$ is a vertex  then every vertex of $(f^*)^{-1}(p)$ is a vertex of $\catsSet(X,D_R(Y))$.
\end{cor}

{
\begin{ex}\label{ex:triangle-same-edges}
{\rm
Let $\overline{\Delta[2]}$ denote the simplicial set obtained by gluing the {$d_1$ and $d_2$} faces of $\Delta[2]$. 
A simplicial distribution $p:\overline{\Delta[2]}\to D(N\ZZ_2)$ can be represented by a tuple $(p^{ab})_{a,b\in \ZZ_2}$, as in Example \ref{ex:triangle}, where {$p^{01}=p^{11}$}.
We will consider the subspace given by the simplicial circle $S^1=\Delta[1]/\partial\Delta[1]$, the {$d_0$}-face of $\overline{\Delta[2]}$.
A simplicial distribution $q:S^1\to D(N\ZZ_2)$ is given by a distribution $q_{\iota_1}\in D(\ZZ_2)$ {(see the notation in  Definition \ref{def:SimpDist})}.
 (Note that the set of simplicial distributions on $(S^1,N\ZZ_2)$ is in bijective correspondence with the set of simplicial distributions on the pair $(\Delta[1],N\ZZ_2)$; see Example \ref{ex:simplex}. {Also observe that $\overline{\Delta[2]}\cong \Delta[0]\ast S^1$  and  $\catsSet(\overline{\Delta[2]},D(N\ZZ_2))\cong \catsSet(S^1,D(\Delta_{\ZZ_2}))$}
under the adjunction in (\ref{eq:adjunction-Dec}).)
The inclusion $f:S^1\to \overline{\Delta[2]}$ as the {$d_0$}-th face induces 
$$
f^*: \catsSet(\overline{\Delta[2]},D(N\ZZ_2))\to \catsSet(S^1,D(N\ZZ_2))\cong [0,1],\;\; (p^{ab}) \mapsto 1-2p^{01}.
$$
The preimage of the vertices $\set{0,1}$ of $[0,1]$ are given by {$(p^{00},0,1-p^{00},0)$ and $(0,1/2,0,1/2)$}; respectively. 
The former component has the deterministic vertices given by $(1,0,0,0)$ and {$(0,0,1,0)$}. Therefore this component coincides with the subset of noncontextual distributions.
The latter component consisting of a single point is a contextual vertex.  
}
\end{ex}
}  

\begin{ex}\label{ex:chsh-boundary}
{\rm
Let $X$ denote the square space in Example \ref{Ex:Sqqq} and $\partial X$ denote its boundary consisting of the nondegenerate $1$-simplices given by $x_i\oplus y_j$ where $i,j\in \ZZ_2$.
The inclusion $f:\partial X\to X$ induces
\begin{equation}\label{eq:f-star}
f^*: \catsSet(X,D(N\ZZ_2))\to \catsSet(\partial X,D(N\ZZ_2))\cong [0,1]^4
\end{equation}
defined by
$$
f^*(p_{x_iy_j}^{ab}) = (p_{x_iy_j}^{00}+p_{x_iy_j}^{11} )_{i,j\in \ZZ_2}.
$$
There are $16$ vertices of the hypercube $[0,1]^4$ given by a tuple $(\delta^{a_{ij}})_{i,j\in \ZZ_2}$ of delta distributions. The preimage is given by (1) a singleton if $\sum_{i,j} a_{ij}=1 \mod 2$, and otherwise (2) the interval $[0,1]$. The $8$ vertices  in case (1) are the PR boxes (hence contextual; see Example \ref{Ex: CHSH}). For example, the preimage of $(\delta^0, \delta^0, \delta^0, \delta^1)$ is given by 
$$
\begin{tabular}{|c|c|c|c|c|} 
\hline
  & $y_0$ & $y_1$ \\ 
\hline
$x_0$ & {\begin{tabular}{cc} $1/2$ & $0$ \\ $0$ & $1/2$ \\ \end{tabular}} & {\begin{tabular}{cc} $1/2$ & $0$ \\ $0$ & $1/2$\\ \end{tabular}} \\ 
\hline
$x_1$ & {\begin{tabular}{cc} $1/2$ & $0$ \\ $0$ & $1/2$ \\ \end{tabular}} & {\begin{tabular}{cc} $0$ & $1/2$ \\ $1/2$ & $0$ \\ \end{tabular}} \\  
\hline
\end{tabular} 
$$
On the other hand, the vertices of case (2) are a pair of deterministic vertices.
}
\end{ex}  
} 

\subsubsection{Homotopic simplicial distributions}

Simplicial homotopy can be used to capture strongly contextual simplicial distributions and vertices.
{We begin with a preliminary observation.}

\begin{lem}\label{SCC}
Let  $f : Z \to X$ be a simplicial set map, and  $p\in \St(X,D_R(Y))$.
\begin{enumerate}
\item $f^\ast(\supp(p)) \subseteq \supp(f^\ast(p))$.
\item If $\supp (f^\ast (p)) \cap f^\ast(\St(X,Y))=\emptyset$ then $p$ is strongly contextual.
\end{enumerate}
\end{lem}
\begin{proof}
Given $\varphi \in \supp(p)$ and $x \in Z_n$, we have 
$$
f^\ast(p)_n(x)(f^\ast(\varphi)_n(x))=(p_n \circ f_n)(x)(\varphi_n \circ f_n (x))=
p_n(f_n(x))(\varphi_n (f_n (x))) \neq 0.
$$
This proves part (1).
If $\varphi \in \supp(p)$ then by part (1) we have 
$f^\ast(\varphi) \in \supp(f^\ast(p))$. This means that  
$f^\ast(\varphi) \in \supp(f^\ast(p)) \cap f^\ast(\St(X,Y))$, proving part (2).
\end{proof}
 
{Part (2) of Lemma \ref{SCC} is a key observation also used in arguments involving simplicial cohomology \cite[Section 5.3]{okay2022simplicial} and \v Cech cohomology  \cite{abramsky2015contextuality}.}

\begin{ex}\label{ex:support-strong-contextual}
{\rm
Consider the inclusion $f:\partial X\to X$ in Example \ref{ex:chsh-boundary}. This map induces
$$
f^*:\catsSet(X,N\ZZ_2) \to \catsSet(\partial X,N\ZZ_2)\cong \ZZ_2^{4}
$$
which sends $\varphi:X\to N\ZZ_2$, determined by the tuple $(\varphi_{\sigma_{ij}})_{i,j\in \ZZ_2}$, to the tuple $(\varphi_{\sigma_{ij}}^{00}+\varphi_{\sigma_{ij}}^{11})_{i,j\in \ZZ_2}$. In particular, the image of $f^*$ consists of $(a_{ij})_{i,j\in \ZZ_2}$ such that 
\begin{equation}\label{eq:tuple-aij}
\sum_{i,j} a_{ij}=0\mod 2.
\end{equation}
 According to Lemma \ref{SCC} a distribution $p\in \catsSet(X,D(N\ZZ_2))$ is strongly contextual if $\supp(f^*(p))$ does not contain    a tuple $(a_{ij})_{i,j}$ satisfying Equation (\ref{eq:tuple-aij}). Note that PR boxes are such that their support satisfy this property, thus they are strongly contextual.
}
\end{ex}

{Let $f_0,f_1:X\to Y$ be two simplicial set maps. {A simplicial homotopy  from $f_0$ to $f_1$ 
is a simplicial set map $F:X\times \Delta[1]\to Y$ such that $f_0=F\circ i_1$ and $f_1=F\circ i_0$,  where $i_k:X \cong X\times \Delta[0] \xrightarrow{\idy\times d^k} X\times \Delta[1]$ for $k=0,1$; see \cite{goerss2009simplicial}.
We will write $f_0\sim f_1$ if there exists a simplicial homotopy from $f_0$ to $f_1$. If no such homotopy exists we will write $f_0\not\sim f_1$.
}

\begin{prop}\label{SCFacet}
Let $\varphi_1, \varphi_2 \in \St(X,Y)$ be such that $\varphi_1\not\sim \varphi_2$.
Then every homotopy $F \in \St(X\times \Delta[1],D_R(Y))$ from 
$\delta_Y \circ \varphi_1$ to  $\delta_Y \circ \varphi_2$ is 
strongly contextual. 
\end{prop}
\begin{proof}
Let  $F \in \St(X\times \Delta[1],D_R(Y))$ be a homotopy from 
$\delta_Y \circ \varphi_1$ to $\delta_Y \circ \varphi_2$. 
Let $i=i_0 \sqcup i_1 : X \sqcup X \to X \times \Delta[1]$.
Then 
$$
\begin{aligned}
\supp (i^\ast (F)) \cap i^\ast(\St(X \times \Delta[1],Y))&=
\supp( (\delta_Y \circ \varphi_1) \sqcup 
(\delta_Y \circ \varphi_2))
\cap i^\ast(\St(X \times \Delta[1],Y)) \\
&=\supp(\delta_Y \circ (\varphi_1 \sqcup \varphi_2))
\cap i^\ast(\St(X \times \Delta[1],Y)) \\
&=\{\varphi_1 \sqcup \varphi_2\}
\cap i^\ast(\St(X \times \Delta[1],Y)) =\emptyset .
\end{aligned}
$$
Here we used the fact that 
$\supp(\delta_Y \circ (\varphi_1 \sqcup \varphi_2))
=\{\varphi_1 \sqcup \varphi_2\}$.
In addition, the last equation follows from the fact that there exists no homotopy from
$\varphi_1$ to $\varphi_2$. 
By Lemma \ref{SCC} part (2) the distribution $F$ is strongly contextual.
\end{proof}

\begin{rem}
{\rm
In the case of $R=\RR_{\geq 0}$ and 
$Y\in s\catGrp$ there is an alternative proof of Proposition \ref{SCFacet} that relies on the equivalence of strong contextuality and strong non-invertibility. By Corollary \ref{SCM} the inclusion $i=i_0 \sqcup i_1 : X \sqcup X \to X \times \Delta[1]$ induces a map of real convex monoids:
$$i^\ast : \St(X \times \Delta[1],D(Y)) \to \St(X\sqcup X,D(Y)).$$ 
According to 
Proposition \ref{Detvert} the deterministic distribution $\delta_Y\circ (\varphi_1\sqcup\varphi_2)$ is a vertex.
In particular, its invertible support consists only of $\delta_Y \circ (\varphi_1 \sqcup \varphi_2)$.
By Proposition \ref{HomInvvv}, the horizontal maps of the following diagram are 
isomorphisms:
$$
\begin{tikzcd}
\St(X \times \Delta[1],Y) \arrow[r,"(\delta_Y)_\ast"] 
\arrow[d,"i^\ast"]  & (\St(X \times \Delta[1],DY))^\ast
\arrow[d,"i^\ast"] \\
\St(X\sqcup X ,Y)  \arrow[r,"(\delta_Y)_\ast"] &
 (\St(X \sqcup X ,DY))^\ast
\end{tikzcd}
$$
Since there is no homotopy from $\varphi_1$ to $\varphi_2$, which means that $\varphi_1 \sqcup \varphi_2 \notin  i^\ast (\St(X \times \Delta[1],Y))$, it turns out that  
$\delta_Y \circ (\varphi_1 \sqcup \varphi_2) \notin
 i^\ast( (\St(X \times \Delta[1],DY)^\ast) $.
We conclude that the intersection of
 $\Isupp (i^\ast (F))$  and $i^\ast( (\St(X \times \Delta[1],DY))^\ast)$ is empty.
 By Corollary \ref{ForAPP1} part (1) 
we obtain that $F$ is strongly non-invertible, and then by Corollary \ref{SC=SNI} part (1) $F$ is strongly contextual.  
}
\end{rem}

\begin{cor}\label{cor:homotopy}
Let $R$ be a 
zero-sum-free, integral semiring and
$\varphi_1, \varphi_2 \in \catsSet(X,Y)$ be such that 
$\varphi_1 \not\sim \varphi_2$. 
If  $F \in \catsSet(X\times \Delta[1],D_R(Y))$  is the unique homotopy from 
$\delta_Y \circ \varphi_1$ to  $\delta_Y \circ \varphi_2$,  then 
$F$ is a strongly contextual vertex in $\St(X\times \Delta[1],D_R(Y))$. 
\end{cor}
\begin{proof}
According to Proposition \ref{Detvert}, the deterministic distribution
 $(\delta_Y \circ \varphi_1) \sqcup 
(\delta_Y \circ \varphi_2)=
\delta_Y \circ (\varphi_1 \sqcup \varphi_2)$ is a vertex 
in $\St(X\sqcup X,D_R(Y))$. By Proposition \ref{f-1vert} 
the unique preimage of  $(\delta_Y \circ \varphi_1) \sqcup 
(\delta_Y \circ \varphi_2)$ under $i^\ast$ is a vertex in
$\St(X\times \Delta[1],D_R(Y))$. This 
vertex is strongly contextual by Proposition
 \ref{SCFacet}.
\end{proof}

\begin{example}\label{HomVertEx}
{\rm
Let $X$ denote the simplicial set obtained by gluing two copies of $\Delta[1]$ along their boundaries:
$$ 
\includegraphics[width=.16\linewidth]{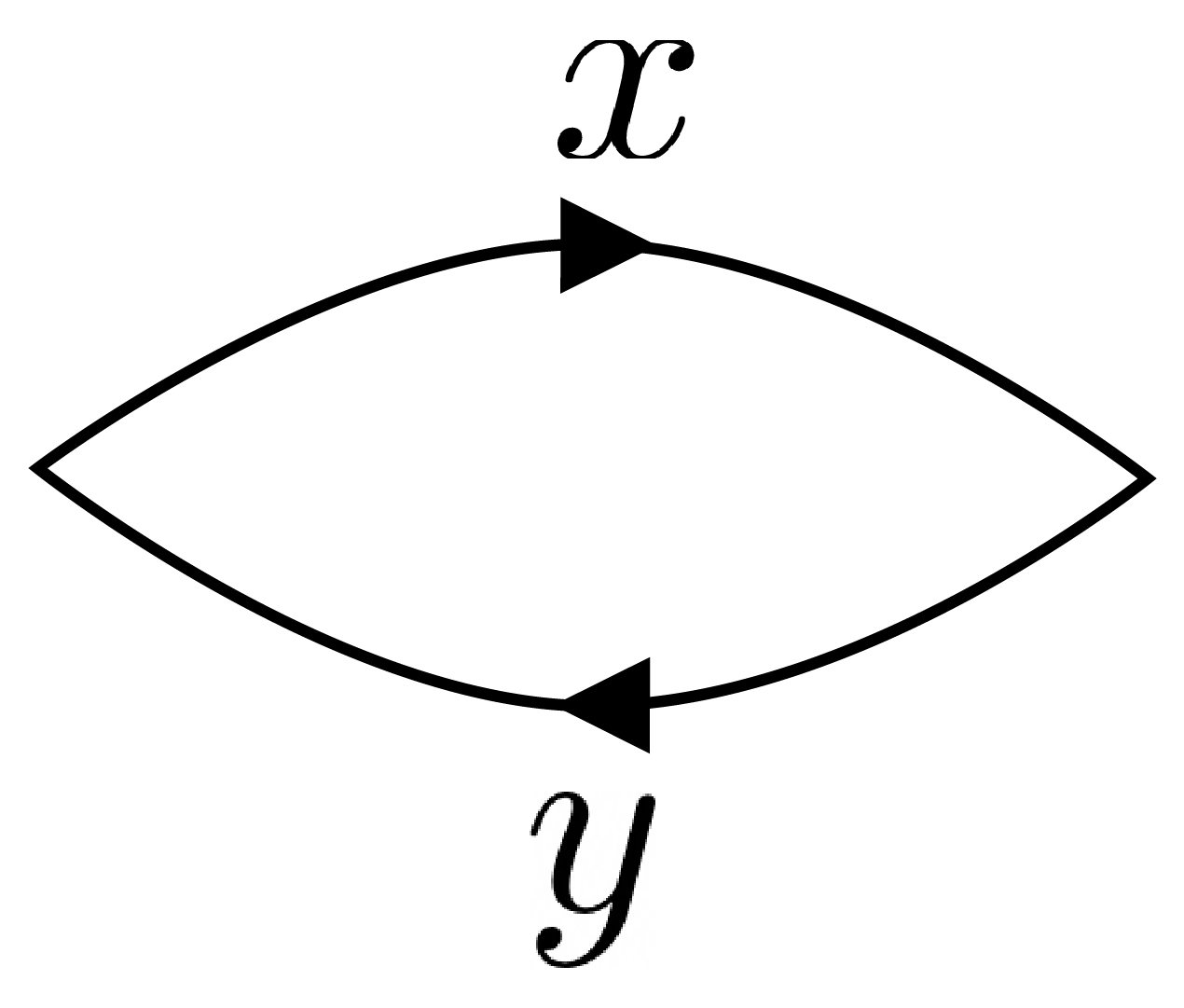}  
$$
Consider two simplicial set maps $\varphi,\psi\in \catsSet(X,N\ZZ_2)$. These maps are determined by the images of the $1$-simplices $x,y$. 
Let $\varphi$ be given by $(x,y) \mapsto (0,0)$, and $\psi$ by $(x,y) \mapsto (1,0)$. Note that $\varphi \not\sim \psi$. But the deterministic distributions $\delta_{N\ZZ_2}\circ \varphi$ and $\delta_{N\ZZ_2}\circ \psi$ are homotopic via the following {unique} homotopy:
$$ 
\includegraphics[width=.4\linewidth]{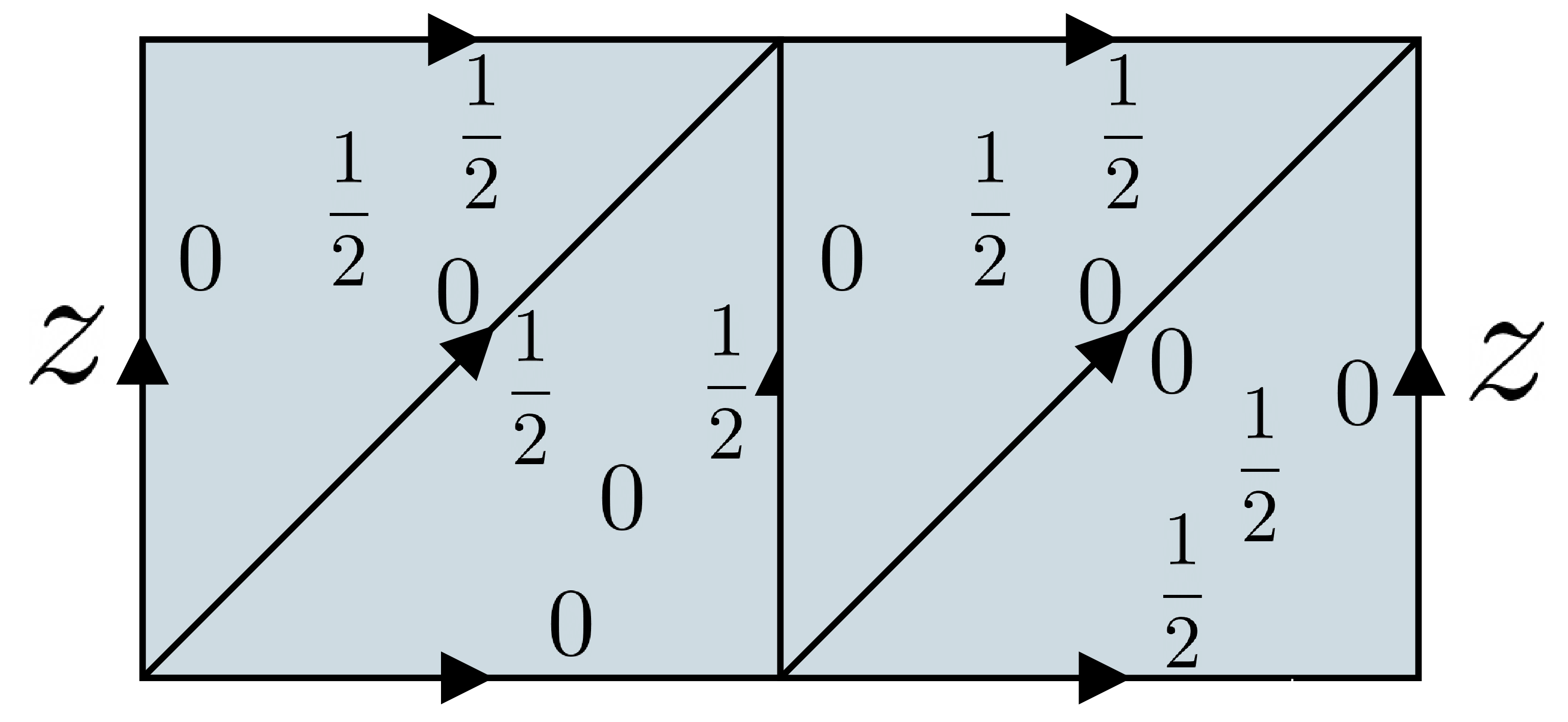}  
$$ 
Note that $(X\times \Delta[1],N\ZZ_2)$ is another way to describe the CHSH scenario of Example \ref{Ex: CHSH}. 
{The distribution above is a PR box. Each of the $8$ distinct PR boxes can be captured in this way as a homotopy.} 

In general, given the four triangles corresponding to the four boxes in the CHSH scenario, any way of assembling them into a simplicial set provides a description as a simplicial distribution. In addition to the realizations given in this example and Example \ref{Ex:Sqqq} there is another realization where $X$ is a punctured torus; see \cite{okay2022simplicial} and \cite{okay2022mermin}.
}
\end{example}



\newcommand{\etalchar}[1]{$^{#1}$}

\end{document}